\documentclass{amsart}
\usepackage[foot]{amsaddr}
\usepackage[margin=1.5in]{geometry}

\usepackage{amsthm,amssymb,amsfonts,amsmath}
\usepackage[T1]{fontenc}
\usepackage{comment}
\usepackage{bbm} 
\usepackage[numbers, sort&compress]{natbib}
\usepackage{graphicx}
\usepackage{hyperref}
\usepackage{color}
\usepackage{enumerate}
\usepackage{mathtools}
\usepackage{ulem}
\usepackage{enumitem}

\providecommand{\R}{}

\providecommand{\N}{}

\renewcommand{\R}{\mathbb{R}}

\renewcommand{\N}{{\mathbb N}}

\newcommand{\set}[1]{\left\{ #1 \right\}}

\newcommand{\Esub}[2]{{\mathbf E_{#1}}\left[#2\right]}

\newcommand\cA{\mathcal A}
\newcommand\cB{\mathcal B}
\newcommand\cC{\mathcal C}

\newcommand\cF{\mathcal F}

\newcommand\cL{{\mathfrak L}} 

\newcommand\cP{\mathcal P}

\newcommand{\bE}{\mathbf{E}}

\newcommand{\bP}{\mathbf{P}} 
\newcommand{\bQ}{\mathbf{Q}} 
\newcommand{\bR}{\mathbf{R}}

\newcommand{\f}{\mathfrak{f}} 
\newcommand{\g}{\mathfrak{g}} 
\newcommand{\h}{\mathfrak{h}} 

\newcommand{\pran}[1]{\left(#1\right)}
\newcommand{\brac}[1]{\left[#1\right]}
\newcommand{\inn}[1]{\langle#1\rangle}
\newcommand{\abs}[1]{\left|#1\right|}
\newcommand{\norm}[1]{\left\|#1\right\|}
\newcommand{\KL}[2]{\mathrm{KL}(#1||#2)}
\providecommand{\eps}{}
\renewcommand{\eps}{\varepsilon}
\providecommand{\ora}[1]{}
\renewcommand{\ora}[1]{\overrightarrow{#1}}

\newcommand\one{\mathbbm{1}}
\newcounter{Scounter}
\newtheorem{thm}{Theorem}
\newtheorem{lem}[thm]{Lemma}
\newtheorem{prop}[thm]{Proposition}
\newtheorem{cor}[thm]{Corollary}
\newtheorem{dfn}[thm]{Definition}

\numberwithin{thm}{section}
\numberwithin{equation}{section}
\theoremstyle{remark}

\makeatletter
\newcommand{\customlabel}[2]{%
   \protected@write \@auxout {}{\string \newlabel {#1}{{#2}{\thepage}{#2}{#1}{}} }%
   \hypertarget{#1}{}
}
\makeatother
\definecolor{clc}{rgb}{0,0.53,0.74}
\definecolor{caz}{rgb}{0,0.5,0}

\begin{document}
\title{Schr\"odinger Bridge Problem for Jump Diffusions}
\author{Andrei Zlotchevski}
\address{Department of Mathematics and Statistics, McGill University, 
Montreal, Canada}
\thanks{We acknowledge the support of the Natural Sciences and Engineering Research Council of Canada (NSERC), awards 559387-2021 and 241023.}
\email{andrei.zlotchevski@mail.mcgill.ca}
\author{Linan Chen}
\email{linan.chen@mcgill.ca}
\begin{abstract} 
    The Schr\"odinger bridge problem (SBP) seeks to find the measure $\hat\bP$ on a certain path space which interpolates between state-space distributions $\rho_0$ at time $0$ and $\rho_T$ at time $T$ while minimizing the KL divergence (relative entropy) to a reference path measure $\bR$. In this work, we tackle the SBP in the case when $\bR$ is the path measure of a jump diffusion. Under mild assumptions, with both the operator theory approach and the stochastic calculus techniques, we establish an $h$-transform theory for jump diffusions and devise an approximation method to achieve the jump-diffusion SBP solution $\hat\bP$ as the strong-convergence limit of a sequence of harmonic $h$-transforms. To the best of our knowledge, these results are novel in the study of SBP. Moreover, the $h$-transform framework and the approximation method developed in this work are robust and applicable to a relatively general class of jump diffusions. In addition, we examine the SBP of particular types of jump diffusions under additional regularity conditions and extend existing results from the diffusion case to the jump-diffusion setting. 
\end{abstract}
\keywords{Schr\"odinger bridges, $h$-transform, jump diffusions, non-local L\'evy-type operators, KL divergence}
\subjclass{35Q93, 45K05, 60H10, 60H20, 60H30, 94A17}
\maketitle
\section{Introduction}
In this introduction, we will cover preliminaries on the long-celebrated Schr\"odinger Bridge Problem (SBP) and the framework of jump-diffusion processes under which our work is conducted. We will also explain our main results and the organization of the paper.
\subsection{Schr\"odinger Bridge Problem: Formulation and Motivation}\label{subsection:SBP motivation}
The original version of this nearly century-old problem was formulated by Erwin Schr\"odinger in \cite{schrodinger1931umkehrung, schrodinger1932theorie}. In Schr\"odinger's thought experiment, a large number of hot gaseous particles are observed at time $t = 0$ to have initial distribution $\rho_0$. The observer's prior belief about the behavior of these particles is that their position should evolve according to a Markov process with transition density $q(s,x,t,y)$ for $0\leq s \leq t \leq T$. As such, by the law of large numbers, the particle density at a later time $T$ should be roughly equal to
\[
\tilde{\rho}_T(\cdot) \approx \int_{\R^n}
q(0, x, T, \cdot )\rho_0(x)dx.
\]
Suppose however that the actual particle density  $\rho_T$ observed at time $T$ vastly differs from the expected distribution $\tilde{\rho}_T$. Schr\"odinger believed that a rare event had occurred and that one must find the ``most probable explanation'' for it. Working from intuition and without a precise statement of the problem, Schr\"odinger nevertheless correctly deduced that the sought explanation lies in the solution to a pair of integral equations, which would be later known as the \textit{Schr\"odinger system}.

More than fifty years after Schr\"odinger's work, F\"ollmer obtained in \cite{follmer1988random} a rigorous formulation of the SBP using the theory of large deviations \cite{dembo1998zeitouni}. To formulate the modern version of the SBP, we need to introduce some terminology. Let $\Omega=D([0,T];\R^n)$ be the collection of all c\`adl\`ag functions $\omega=(\omega(t))_{0\leq t\leq T}$ from the time interval $[0,T]$ to $\R^n$. This is the space of trajectories, and $\omega\in \Omega$ is a \textit{sample path}. Let $\Omega$ be equipped with the Skorokhod metric.
The resulting Borel $\sigma$-algebra $\cF:=\cB_{\Omega}$ is generated by the time projections $X_t:\omega\in\Omega\mapsto\omega(t)\in \R^n$ for $t\in[0,T]$. Denote by $\cP(\Omega)$ the set of probability measures on $(\Omega,\cF)$. We will refer to $\bR\in\cP(\Omega)$ as a \textit{path measure} and $\set{X_t}_{[0,T]}$ the stochastic process given by the time projections as the \textit{canonical process} under $\bR$. Let $\set{\cF_t}_{[0,T]}$ be the natural filtration generated by $\set{X_t}_{[0,T]}$ and, whenever necessary, we will adopt the \textit{usual augmentation} \cite[II-67]{rogers2000diffusionsVOL1} of this filtration (by $\bR$ or another path measure), which we continue writing as $\set{\cF_t}_{[0,T]}$.

Let $\omega_1,...,\omega_N \in \Omega$ be independently sampled paths according to a ``reference'' path measure $\bR\in\cP(\Omega)$, and denote their empirical distribution by $Z_N$. Then as $N\to\infty$, Sanov's theorem \cite{sanov1961probability} gives the (asymptotic) likelihood that the empirical distribution of these paths deviates from its true distribution $\bR$:
\begin{equation*}
    \text{Prob}(Z_N\in A) \approx \exp(-N\inf_{\bP\in A} \KL{\bP}{\bR}) \text{ for } A \subseteq \cP(\Omega),
\end{equation*}
where $\KL{\cdot}{\cdot}$ is the \textit{Kullback-Leibler divergence} (KL divergence), also known as the \textit{relative entropy}. For probability measures $\mu$ and $\nu$, it is defined to be
\begin{equation*}
    \KL{\mu}{\nu} := \begin{cases}
        \displaystyle \Esub{\mu}{\log\pran{\frac{d\mu}{d\nu}}} & \text{if }\mu \ll \nu,\\
        +\infty & \text{otherwise.}
    \end{cases}
\end{equation*}
In other words, if a large deviation event described by $A$ has occurred (such as particles having distribution $\rho_T$ instead of $\tilde{\rho}_T$), finding the ``most probable explanation'' envisioned by Schr\"odinger corresponds to finding the minimizer in the decay exponent $\inf_{\bP\in A} \KL{\bP}{\bR}$. Therefore, the modern formulation of the SBP aims to bridge the two endpoint distributions $\rho_0,\rho_T$ while staying as close as possible (in the sense of KL divergence) to the reference path measure $\bR$.\\

For a path measure $\bP\in\cP(\Omega)$, let $\bP_t\in \cP(\R^n)$ be its marginal distribution at time $t$, meaning that for any Borel $B\subseteq\R^n$, $\bP_t(B)=\bP(\set{X_t\in B})$. The modern formulation of the SBP states:
\begin{dfn}[Schr\"odinger Bridge Problem]
Given a reference path measure $\bR\in \cP(\Omega)$ and target endpoint distributions $\rho_0,\rho_T \in \cP(\R^n)$, determine
\begin{equation}\label{Definition:DynamicSBP}
    \hat{\bP}=\arg\min_{\bP \in \cP(\Omega) }\set{\KL{\bP}{\bR} \text{ such that } \bP_0=\rho_0,\bP_T=\rho_T }.
\end{equation}

We say that the SBP does not have a solution if the ``\textup{min}'' in \eqref{Definition:DynamicSBP} is $\infty$.
\end{dfn}
The optimizer $\hat\bP$ is called the \textit{Schr\"odinger bridge} between $\rho_0$ and $\rho_T$.
Early work on the SBP considered the reference measure $\bR$ being the law of a standard Brownian motion and later generalized to the case when $\bR$ arises from an $\R^n$-valued diffusion process that solves\footnote{In this work, only the weak solution of SDE is concerned.} a general \textit{It\^o Stochastic Differential Equation} (SDE)
\[
dX_t=b(t,X_t)dt+\sigma(t,X_t)dB_t,
\]
where $\set{B_t}_{t\geq0}$ is a standard $m$-dimensional Brownian motion. In the present work, we extend this framework to consider diffusion processes with jumps. More specifically, let $\nu$ be a $\sigma$-finite measure on $\R^\ell$ that satisfies\footnote{Throughout, the notation ``$a\wedge b$'' refers to $\min\{a,b\}$.}
\begin{equation*}
    \nu(\set{0})=0 \text{ and } \int_{\R^\ell} (1 \wedge |z|^2) \nu(dz) <\infty.
\end{equation*} 
We consider $\bR\in\cP(\Omega)$ being the law of the solution to the following \textit{L\'evy-It\^o SDE}:
\begin{equation}
\label{ReferenceLevyIto}
 \begin{aligned}   dX_t=b(t,X_{t})&dt+\sigma(t,X_{t})dB_t \\&+ \int_{|z|\leq 1}\gamma(t,X_{t-} ,z)\tilde{N}(dt,dz)
 +
    \int_{|z|>1}\gamma(t,X_{t-} ,z)N(dt,dz),
\end{aligned}
\end{equation}
where $\set{B_t}_{t\geq 0}$ is a standard $m$-dimensional Brownian motion, $\{N(dt,dz)\}_{t\geq 0}$ is an $\ell$-dimensional jump measure associated with a Poisson process with intensity $\nu(dz)$, $\tilde{N}(dt,dz)$
is the compensated jump measure, i.e., $\tilde{N}(dt,dz):= N(dt,dz)-\nu(dz)dt$, and $\set{B_t}_{t\geq 0}$ and $\set{N(dt,dz)}_{t\geq 0}$ are independent. The coefficients $\sigma$, $b$ and $\gamma$ are required to satisfy some mild conditions, which will be stated later, to guarantee that all the stochastic integrals in \eqref{ReferenceLevyIto} are well defined.
The solution to the SDE \eqref{ReferenceLevyIto}, whenever it exists, is called a \textit{jump diffusion}, also called \textit{jump-diffusion process} or \textit{diffusion process with jumps} in the literature. To the best of our knowledge, the SBP for general jump diffusions has not been comprehensively studied.

In addition to the SDE \eqref{ReferenceLevyIto} representation, we can also describe the reference measure $\bR$ through its associated time-dependent integro-differential operator $L$ as, for $f\in C^{1,2}(\R^+\times \R^n)$,
\begin{equation}\label{ReferenceGenerator}
\begin{aligned}
    Lf(t,x) :=& \,b(t,x)\cdot\nabla f (t,x)+\frac{1}{2}\sum_{i,j}(\sigma\sigma^T)_{ij}(t,x)\frac{\partial^2f}{\partial x_i \partial x_j}(t,x) \\
    &+\int_{\R^\ell} \brac{f(t,x+\gamma(t,x,z))-f(t,x)-\one_{|z|\leq 1}\gamma(t,x,z)\cdot\nabla f(t,x)}\nu(dz).    
\end{aligned}
\end{equation}
Throughout this work, the symbol $\nabla$ (or $\nabla^2$) denotes the gradient (or the Hessian matrix) with respect to the spatial variable $x$ only. We call an operator of the form \eqref{ReferenceGenerator} a \textit{jump-diffusion operator} and we refer to $L$ as the \textit{generator} of $\bR$, or equivalently, the generator of the SDE \eqref{ReferenceLevyIto}. We also define $\cL:=\frac{\partial}{\partial t}+L$.
Throughout this work, we impose the following conditions on the coefficients $\sigma$, $b$ and $\gamma$:
\paragraph{\textbf{Assumptions (B)}}\customlabel{Assumptions B}{\textbf{(B)}}
\begin{enumerate}
    \item 
    $b=(b_i)_{1\leq i\leq n}:\R^+\times\R^n \to \R^n$ is measurable and locally bounded.
    \item 
    $\sigma=(\sigma_{ij})_{1\leq i\leq  n,1\leq j \leq m} : \R^+\times\R^n \to\R^{n\times m}$ is measurable and locally bounded, and $(\sigma\sigma^T)(t,x)$ is non-negative definite for every $(t,x)\in\R^+\times\R^n$.
    \item 
    $\gamma=(\gamma_i)_{1\leq i\leq n} :\R^+\times\R^n\times\R^\ell \to \R^n$ is measurable, $\gamma(t,x,0)\equiv 0$ for every $t\geq 0$ and $x\in\R^n$, and, as a function in $(t,x)$, $\sup_{z\in \mathbb{R}^\ell}\frac{|\gamma(t,x,z)|}{1\wedge|z|}$ is locally bounded.
\end{enumerate}
The above conditions are rather mild, merely imposed for the well-definedness of the integrals in \eqref{ReferenceLevyIto} and \eqref{ReferenceGenerator}, as well as of other (stochastic) integrals we will encounter later in this study. In fact, as we will see in Section \ref{subsection:main contribution}, our main results are established under another set of assumptions that are not stated directly in terms of $\sigma,b,\gamma$. In this way, we aim at establishing a general theory that is not restricted to SDEs or generators with any particular type of coefficients.

\subsection{Main Results}\label{subsection:main contribution}
We consider the SBP \eqref{Definition:DynamicSBP} for reference measure $\bR$ being the law of the jump diffusion associated with the SDE  \eqref{ReferenceLevyIto} or the operator \eqref{ReferenceGenerator} and target endpoint distributions $\rho_0,\rho_T$. As we will review in Section \ref{Section:DiffusionReview}, when $\bR,\rho_0,\rho_T$ satisfy some mild conditions (see Theorem \ref{Theorem:fg-existence} for details), \eqref{Definition:DynamicSBP} admits a unique solution $\hat\bP$. Our goal is to conduct a comprehensive study of $\hat\bP$ in the jump-diffusion setting and to develop necessary theoretical elements and technical tools for the interpretation and the analysis of $\hat\bP$. The main challenge we face is the scarcity of the regularity theory for jump diffusions in comparison with diffusions. Many important properties (e.g., hypoellipticity and heat kernel estimates) are far less understood for general jump diffusions. As a consequence, the extension of the literature on SBP from the diffusion case to the jump-diffusion setting requires careful treatment. 

It will also become clear in the review section that  $\hat\bP$ is closely tied to the notion of \textit{$h$-transform}. Therefore, to interpret $\hat\bP$ properly, we first develop the theory of $h$-transform for general jump diffusions. The topic of $h$-transform is of independent interest on its own. Then, upon establishing $\hat\bP$ in the context of $h$-transform, we further investigate conditions under which $\hat\bP$ possesses similar properties as its counterpart in the diffusion case. We also propose an approximation approach toward $\hat\bP$ under mild regularity requirements.\\

\noindent \textbf{Assumptions and Results:} We work closely but separately with the SDE \eqref{ReferenceLevyIto} and the operator \eqref{ReferenceGenerator}, as they offer two different angles through which we can examine the SBP and its solution $\hat\bP$. Furthermore, instead of putting concrete constraints on the coefficients, we adopt basic assumptions that a large family of SDEs and operators are expected to satisfy. 

We start with the assumption on the operator, as this is our most general assumption. We assume that the jump-diffusion operator $L$ in \eqref{ReferenceGenerator} satisfies the following:
\begin{enumerate}[label=(\textbf{A1})]
    \item \label{Assumption:MartingaleProblem-s,x} The martingale problem for $L$ has a strong Markov solution in $\cP(\Omega)$. Namely, for every $(s,x)\in[0,T]\times\R^n$, there exists a path measure $\bP^{s,x}\in\cP(\Omega)$ such that $\bP^{s,x}(\set{X_s=x})=1$ and for every $\phi \in C_c^{1,2}(\R^+ \times \R^n)$,
    \begin{equation*}
        \set{\phi(t,X_t)-\phi(s,x)-\int_s^t(\frac{\partial}{\partial r}+L)\phi(r,X_r)dr \,:\,s\leq t \leq T}
    \end{equation*}
    under $\bP^{s,x}$ is a martingale with respect to $\set{\cF_t}_{[s,T]}$; further, the family $\set{\bP^{s,x}}_{[0,T]\times\R^n}$ is strong Markovian in the sense that for every $s\in[0,T]$ and stopping time $\tau\in[s,T]$, the regular conditional probability distribution of $\bP^{s,x}$ conditioning on $\cF_\tau$ is given by $\bP^{\tau,X_\tau}$. 
\end{enumerate}
Note that under Assumption \ref{Assumption:MartingaleProblem-s,x}, given any $\rho_0\in \cP(\R^n)$, there exists a path measure $\bP\in\cP(\Omega)$ such that $\bP_0=\rho_0$, and $\set{X_t}_{[0,T]}$ under $\bP$ is a strong Markov process with transition distribution 
    \begin{equation*}
        P_{s,t}(x,dy):=\bP(X_t \in dy | X_s=x)=\bP^{s,x}(X_t\in dy)
    \end{equation*} 
    for every $0 \leq s \leq t \leq T$ and $x\in\R^n$. In this case we also refer to $\bP$ as a \textit{solution to the martingale problem} for $(L,\rho_0)$.

Assumption \ref{Assumption:MartingaleProblem-s,x} is already sufficient for us to establish a basic $h$-transform theory. In Section \ref{Subsection:Generator Approach}, under \ref{Assumption:MartingaleProblem-s,x}, we propose a definition (see Definition \ref{Definition:h-transform,general case with r0}) of the $h$-transform of $\bP$, formally as $\bP^h:=\frac{h(T,X_T)}{h(0,X_0)}\bP$, for $\bP$ being the solution to the martingale problem for $(L,\rho_0)$ and $h$ being a non-negative function that satisfies the \textit{mean-value property} (see Definition \ref{def:meanvalueproperty}). Our main result for this part is \textbf{Theorem \ref{Theorem:MartingaleProblem-Lh}} where we prove that, if in addition $h$ is a \textit{harmonic function} (see Definition \ref{Definition:HarmonicFunction}), then $\bP^h$ solves the martingale problem for another jump-diffusion operator $L^h$ equipped with an appropriate initial endpoint distribution $\rho^h_0$. We also explicitly determine $L^h$ as mapped from $L$ through a translation in the drift coefficient and a rescaling in the jump component.\\

Next, we turn to the SDE approach and adopt the following assumption:
\begin{enumerate}[label=(\textbf{A2})]
    \item \label{Assumption:SDEsolution} The SDE \eqref{ReferenceLevyIto} with initial condition $\rho_0$ admits a solution that is a strong Markov process with c\`adl\`ag sample paths.  To be accurate, there exists a filtered probability space $(\tilde\Omega,\tilde\cF,{\{\tilde\cF_t\}},\tilde\bP)$ on which \eqref{ReferenceLevyIto} is satisfied by $\{\tilde\cF_t\}$-adapted $\{\tilde X_t,B_t,N(dt,dz)\}$ such that $\set{B_t}_{[0,T]}, \set{N(dt,dz)}_{[0,T]}$ are as described in Section \ref{subsection:SBP motivation} and $\{\tilde X_t\}_{[0,T]}$ is a strong Markov process with c\`adl\`ag sample paths and $\tilde X_0\sim\rho_0$.
\end{enumerate}
Under \ref{Assumption:SDEsolution}, since $\{\tilde X_t\}_{[0,T]}$ has sample paths in $\Omega$, its law under $\tilde\bP$ yields a path measure $\bP\in\cP(\Omega)$ whose canonical process $\set{X_t}_{[0,T]}$ is again strong Markovian. With a slight abuse of notation, we will identify $(\{\tilde X_t\}_{[0,T]},\tilde \bP)$ with $(\set{X_t}_{[0,T]},\bP)$, and refer to $(\set{X_t}_{[0,T]},\bP)$ as the solution to \eqref{ReferenceLevyIto}. 

Assumption \ref{Assumption:SDEsolution} enables us to re-examine $\bP^h$, the $h$-transform of $\bP$, with stochastic calculus tools. In particular, under the same assumption that $h$ is non-negative, harmonic and with the mean-value property, we establish in \textbf{Theorem \ref{GirsanovHTransform}} that $\bP^h$ can be obtained through a Girsanov transformation from $\bP$, and we further state in \textbf{Corollary \ref{Theorem:SDEunderPh}} that the canonical process under $\bP^h$, along with properly constructed Brownian motion and jump measure, is the solution to a new L\'evy-It\^o SDE in the form of (\ref{ReferenceLevyIto}) whose coefficients match those of $L^h$ derived in the generator approach. This confirms that the interpretations of $\bP^h$ from the generator approach and the SDE approach are consistent with one another. \\

In Section \ref{Section:RelationSBP-htransform}, we return to the SBP for jump diffusions. In order to connect the SBP solution $\hat\bP$ with the theory of $h$-transforms, in addition to \ref{Assumption:MartingaleProblem-s,x} (or \ref{Assumption:SDEsolution}), we also impose the following assumption: 
\begin{enumerate}[label=(\textbf{A3})] 
    \item \label{Assumption:Hypo} Let $P_{s,t}(x,\cdot)$, $0\leq s\leq t\leq T$, $ x\in\R^n$, be the transition distribution of $\set{X_t}_{[0,T]}$ under $\bR$. For any $g \in C_c^\infty(\R^n)$, 
    if $h:[0,T]\times\R^n\to\R$ is the function given by 
    \begin{equation}\label{eq:integral def of harmonic h by g in Cc}
         h(t,x):=\int_{\R^n}g(y)P_{t,T}(x,dy)\text{ for }(t,x)\in[0,T]\times \R^n,
    \end{equation}
    then $h$ is \textit{of class $C^{1,2}$} on $[0,T]\times\R^n$ in the sense that $h,\nabla h,\nabla^2h\in C([0,T]\times\R^n)$ and $\frac{\partial}{\partial t}h\in C((0,T)\times \R^n)$. Moreover, $h$ satisfies $(\frac{\partial}{\partial t}+L) h=0$ on $(0,T)\times\R^n$. 
\end{enumerate}
What \ref{Assumption:Hypo} in fact entails is that, for every (non-negative) $g\in C^\infty_c(\R^n)$, $h$ defined in (\ref{eq:integral def of harmonic h by g in Cc}) is a (non-negative) harmonic function with the mean-value property. Therefore, the integral in (\ref{eq:integral def of harmonic h by g in Cc}) produces natural ``candidates'' for performing $h$-transforms.

Back to the SBP \eqref{Definition:DynamicSBP} for $\bR,\rho_0,\rho_T$, as we will review in Section \ref{Section: static SBP and fg-system}, the solution $\hat\bP$ can be written as $\hat\bP=\frac{\h(T,X_T)}{\h(0,X_0)}\bP$ where $\bP$ is a path measure derived from $\bR$ by only altering the initial endpoint distribution, and $\h$ is defined as in (\ref{eq:integral def of harmonic h by g in Cc})  with some non-negative measurable function $g$ that is entirely determined by $\bR_{0T},\rho_0,\rho_T$.  Thus, formally speaking, we can identify $\hat\bP$ as $\bP^\h$, the $h$-transform of $\bP$ by $\h$. However, such an $\h$ may not be harmonic, which means that the aforementioned $h$-transform theory cannot be applied directly to $\bP^\h$. Instead,  we establish an approximation procedure for $\bP^\h$ (and $\hat\bP$). In \textbf{Theorem \ref{thm:P^hk strong conv to hatP}}, under Assumption \ref{Assumption:Hypo}, we prove that there exist non-negative harmonic functions $\set{h_k: k\geq 1}$ that approximate $\h$ in some proper sense, such that $\bP^{h_k}$, the corresponding $h$-transform of $\bP$ by $h_k$, converges to $\hat\bP$ in the strong topology. Moreover, the construction of such $h_k$'s is carried out through (\ref{eq:integral def of harmonic h by g in Cc})  with $g_k$'s only dependent on $\bR_0,\bR_T,\rho_0,\rho_T$, which makes this approximation procedure universal for jump diffusions (provided that \ref{Assumption:Hypo} is supported) and hence convenient to implement in practice.\\

In Section \ref{Section:smooth-coefficients}, we examine a concrete case where  \ref{Assumption:MartingaleProblem-s,x},  \ref{Assumption:SDEsolution}, and \ref{Assumption:Hypo} are all satisfied. In particular, we adopt the setup from \cite{kunita2019stochastic} to consider the SBP for L\'evy-It\^o SDEs with \textit{smooth coefficients} and \textit{diffeomorphic jumps}. Under this framework, the solution to the SDE not only exists but also admits a transition density with desirable regularity. 
In \textbf{Theorem \ref{THeorem:varphi_hatvarphi solution to Schrodinger system}}, we formulate the \textit{dynamic Schr\"odinger system} associated with the SBP in terms of $(\varphi,\hat\varphi)$, a pair of functions of class $C^{1,\infty}$ on $[0,T]\times\R^n$ that solve the corresponding Kolmogorov backward and forward equations respectively. The pair $(\varphi,\hat\varphi)$ in turn offers further information on the solution $\hat\bP$ to the SBP, as the marginal density of $\hat\bP$ is exactly given by the product $\varphi\hat\varphi$. This result extends the existing literature on the diffusion SBP to the jump-diffusion setting.

At last, we treat the SBP for a common type of jump-diffusion operator, and that is when the operator $L$ has an $\alpha$-stable-like jump component, such as the model studied in \cite{czq2016heatkernel}. Although such an $L$ may not support \ref{Assumption:Hypo}, it is possible to approximate $L$ in some proper sense by jump-diffusion operators that fulfill \ref{Assumption:Hypo} and hence yield ``nice-behaving'' $h$-transforms. It is natural to expect that these $h$-transforms should be ``close'' to the solution $\hat\bP$ to the SBP for $L$. Indeed, we prove in \textbf{Theorem \ref{Thm:approximation of P^ in finite dimensional dist}} that it is possible to construct a sequence of path measures, all of which are $h$-transforms by harmonic functions, that converges to $\hat\bP$ in all finite-dimensional distributions. \\

\noindent \textbf{Generality of the Assumptions:} Let us give a brief discussion of our assumptions. Assumption \ref{Assumption:MartingaleProblem-s,x} is satisfied when the \textit{martingale problem} for $L$ is well posed. The well-posedness of martingale problem for jump-diffusion operators has long been investigated and has generated a rich literature. For example, for general $L$, it has been shown in \cite{Stroock1975,komatsu1973markov} that the martingale problem is well posed if, as functions in $(s,x)$, $\sigma(s,x)$ and $\int_{\R^\ell}|\gamma(s,x,z)|^2\nu(dz)$ are bounded and continuous, and $b(s,x)$ is bounded and measurable. More recently, for $L$ with a stable-like jump component, the martingale problem has also been extensively studied \cite{chen2021periodic,kuhn2020existence,chen2016heat,chen2020heat,chen2016uniqueness,jin2017heat}.

In fact, in the case of diffusions, \ref{Assumption:MartingaleProblem-s,x} also implies \ref{Assumption:SDEsolution} as it is a well-known result that the existence of a martingale problem solution for a diffusion operator is equivalent to that of a weak solution to the corresponding It\^o SDE \cite{stroock1997multidimensional,rogers2000diffusionsVOL2}. However, such an equivalence is not fully established yet for general jump diffusions; some results in this direction are known in the case of time-homogeneous coefficients \cite{lepeltier1976probleme, kurtz2011equivalence}. On the other hand, the SDE \eqref{ReferenceLevyIto} is known to admit a unique \textit{strong} solution under Lipschitz-type conditions on the coefficients (e.g., \cite[Chapter 6.2]{applebaum_2009}), in which case \ref{Assumption:SDEsolution} will be fulfilled.

Assumption \ref{Assumption:Hypo} is known to hold for diffusions under mild conditions on the coefficients, but analogous results are far less abundant in the jump-diffusion case. A typical scenario in which \ref{Assumption:Hypo} holds is when the transition distribution $P_{s,t}(x,\cdot)$ admits a density $p(s,x, t,\cdot)$ that possesses the desired regularity, and thus the regularity of $h$ can be inherited from $p$. Such is the case for the two classes of jump diffusions that we will cover in Section \ref{Section:smooth-coefficients}. As the literature of jump diffusions is fast developing, we expect that \ref{Assumption:Hypo} will be shown to hold in more situations.\\

\noindent \textbf{Novelty of the Work:} Although the SBP for diffusions has been extensively treated and well understood, rigorous studies of the counterpart problem for jump diffusions are still scarce, and a comprehensive mathematical theory for the jump-diffusion SBP is still missing. This work is a step toward filling the gaps. More generally, our work is aligned with the fast-growing body of research that connects/compares diffusions with jump diffusions, and local operators with non-local operators. Below we explain some novelties of this work:

(1) We adopt a general framework that encompasses a wide range of diffusions and jump diffusions. Instead of restricting ourselves to cases with concrete conditions on the coefficients, we build the theory on some basic properties (i.e., \ref{Assumption:MartingaleProblem-s,x} or \ref{Assumption:SDEsolution}, \ref{Assumption:Hypo}) that are known to hold for a large class of diffusions and jump diffusions. This includes models whose coefficients  may not be ``nice-behaving'', e.g., certain operators with degenerate diffusion components or singular drifts. 

(2) In our approach of solving the SBP, we keep separate the ``impact'' from the target endpoint distributions $\rho_0,\rho_T$ and that from the ``nature'' of the operator $L$, and demonstrate that the latter holds a more essential role in configuring the solution $\hat\bP$. It is clear from our results that, for a ``good-natured''  $L$ (i.e., supporting the assumptions), we can always access $\hat\bP$ through $h$-transforms, either as a harmonic-$h$-transform (which recovers the classical picture of the diffusion SBP), or as a strong-convergence limit of harmonic-$h$-transforms (which presents a novel approach in the study of $\hat\bP$). 

(3) We develop methodologies that are robust and suitable for the study of jump-diffusion SBPs when the regularity theory is insufficient. For example, the $h$-transform theory proposed here does not require $h$ to be strictly positive and is naturally adapted to the zero sets of $h$. The approximation procedure for $\hat\bP$ is also given in terms of explicit formulas and tractable constructions, which makes it more accessible for general jump diffusions.
\subsection{Connections between SBP and Other Fields}
Schr\"odinger's original motivation for posing his problem was based in quantum physics. Over the years, as the theory behind Schr\"odinger bridges developed, it has seen applications in various fields. The discovery of the connection between SBP and \textit{Optimal Transport} (OT), as well as the stochastic control formulation of SBP, prompted a new wave of interest. While the field of OT is fundamentally concerned with finding the mapping between distributions $\rho_0$, $\rho_T$ that minimizes a cost function, the SBP inherently provides a way to reconstruct the dynamics of the process as it evolves from $\rho_0$ to $\rho_T$. As for the stochastic control formulation of the SBP, in the setting of diffusion processes, it converges to the Benamou-Brenier (fluid dynamic) formulation of OT as the diffusion coefficient tends to $0$. These rich connections have been extensively studied \cite{dai1991stochastic, mikami2004monge, leonard2013survey, chen2021optimal}.

In the last decade, these topics have attracted intense attention due to their applications in \textit{machine learning} \cite{wang2021deep, de2021diffusion(MachineLearning),liu2022deep,tzen2019theoretical,pavon2021data,shi2024diffusion,vargas2021solving}, where diffusion SBP models are steadily rising in popularity. To cite some examples, diffusion Schr\"odinger bridges have been adopted for generative modeling tasks \cite{de2021diffusion(MachineLearning)}, their theoretical properties in neural networks have been investigated in \cite{tzen2019theoretical}, and their numerical approximations have been developed in \cite{vargas2021solving}. 
While having a vast range of applications, diffusion models also have limitations as they only describe dynamics that are continuous in time. However, in the study of certain phenomena, it is beneficial or even necessary to adopt models that accommodate ``jumps'', for which \textit{L\'evy processes} are natural candidates.  
For example, due to their nature of combining continuous and jump dynamics, L\'evy processes have been widely adopted in applications to finance as more realistic financial models \cite{eberlein2001application, schoutens2003levy}. 
More recently, jump diffusions\footnote{In finance literature, the word \textit{jump diffusion} often refers to the specific case where the jump component is a compound Poisson process.}, as generalizations of L\'evy processes, have also drawn increasing interest from the mathematical finance community as potential models for the prices of assets, options, and other financial derivatives, as well as interest rate models \cite{cont2005finite, kou2007jump,michelbrink2012martingale,bouzianis2021finance}. 
As an extension of both classical diffusions and L\'evy processes, jump diffusions are promising tools in creating richer models in applied scientific fields. Building a comprehensive theory for the jump-diffusion SBP will not only expand the existing literature, but also shed light on potential new areas of applications. 

\subsection{Further Questions} In this paper, we provide a detailed analysis of the SBP for jump diffusions from the operator and the SDE perspectives. In fact, there exists a third angle, the \textit{stochastic control} formulation  \cite{dai1991stochastic}, which is well-established in the diffusion case. In an ongoing work, we are working on the stochastic control approach to the jump-diffusion SBP. 
 
We also intend to revisit the SBP for processes with killing (known as the \textit{unbalanced SBP})\cite{chen2021usbp} and to interpret it as a jump-diffusion SBP, using both the SDE approach and the stochastic control formulation. 
More generally, we believe that the jump-diffusion framework allows one to reinterpret many SBP-type problems for (jump) diffusions which possess an additional complexity such as a regime-switching mechanism. 
 
In addition, since the SBP and the Schr\"odinger system are symmetric with respect to switching the marginals, it is natural to consider the dynamics of the Schr\"odinger bridge under time reversal. A combination of both forward and backward dynamics has already been applied in some studies of diffusion SBP \cite{chen2021optimal}. The theories and the techniques developed in this work have the potential to facilitate a similar study for jump diffusions.

On the application side, it is known that, when the reference process admits a positive and continuous transition density, the SBP solution can be obtained via an efficient algorithm (see Theorem \ref{Theorem:FBJ-Existence} and the discussion below). While this is the case for a large class of diffusions, only specific types of jump diffusions have been proven to possess densities. Therefore, it is worth exploring whether an effective algorithm can be devised to solve SBPs with relaxed density requirements. In a similar spirit, the special case of the unbalanced SBP for diffusions is studied in \cite{chen2021usbp}, but results in more general circumstances are missing.

\section{Static SBP and Diffusion SBP}\label{Section:DiffusionReview}
In this section, we will review some important aspects of the theory of the static SBP and the literature of the diffusion SBP. The former will serve as the ``start point'' of our investigation, and the latter will describe the ``picture'' we aim to reproduce for jump diffusions. 
\subsection{Static SBP and Schr\"odinger System}\label{Section: static SBP and fg-system}
Schr\"odinger's problem (\ref{Definition:DynamicSBP}) is also referred to as the \textit{dynamic} SBP because it involves a minimization over the space of path measures. In many situations, it is easier to solve a \textit{static} version of the problem, and to relate the static solution back to the dynamic formulation. 
\begin{dfn}[Static Schrodinger Bridge Problem]
Given a reference measure $R\in \cP(\R^n\times \R^n)$ and target marginals $\rho_0,\rho_T \in \cP(\R^n)$, determine
\begin{equation}\label{Definition:StaticSBP}
    \hat{\pi}=\arg\min_{\pi \in \cP(\R^n\times \R^n)}\set{\KL{\pi}{R} \text{ such that }  \pi_0=\rho_0,\pi_T=\rho_T }
\end{equation}
where $\pi_0,\pi_T$ are respectively the marginal distributions of $\pi$ in the first and second variable.

We say that the static SBP does not have a solution if the ``\textup{min}'' in \eqref{Definition:StaticSBP} is $\infty$.
\end{dfn}
To relate the static SBP to the dynamic problem (\ref{Definition:DynamicSBP}), let us begin by introducing some notation. For a path measure $\bP\in\cP(\Omega)$ and $0\leq s\leq t\leq T$, define $\bP_{st}$ to be its two-dimensional joint marginal distribution at times $s$ and $t$.
Moreover, if for $x,y\in\R^n$,
\begin{equation*}
    \bP^{xy}:=\bP(\cdot|X_0=x,X_T=y)
\end{equation*}
is the induced \textit{bridge measure} between $x$ and $y$, then we have the decomposition of $\bP$ as
\begin{equation}\label{Equation:Disintegration}
    \bP(\cdot)=\int_{\R^n\times\R^n}\bP^{xy}(\cdot)\bP_{0T}(dxdy).
\end{equation}
For a given choice of $\rho_0,\rho_T$, denote the collection of their couplings by 
\begin{equation*}
    \Pi(\rho_0, \rho_T) := \set{\pi\in\cP(\R^n\times\R^n) : \pi_0=\rho_0, \pi_T=\rho_T}. 
\end{equation*}
A path measure $\bP\in\cP(\Omega)$ is called \textit{admissible} if $\bP_{0T}\in\Pi(\rho_0,\rho_T)$. 
\begin{thm}\cite{follmer1988random}\label{thm:relation static and dynamic SBP}
The Schr\"odinger problems \eqref{Definition:DynamicSBP} and \eqref{Definition:StaticSBP} admit at most one solution $\hat{\bP}$ and $\hat{\pi}$ respectively. Assume that the reference measure $R$ in the static problem \eqref{Definition:StaticSBP} and the reference path measure $\bR$ in the dynamic problem \eqref{Definition:DynamicSBP} are such that $R=\bR_{0T}$. Then, we have the following relations between the solutions: 
    \begin{enumerate}
    \item If $\hat{\pi}$ is the solution to \eqref{Definition:StaticSBP}, then
            \begin{equation}\label{StaticToDynamic}
                \hat{\bP}(\cdot):=\int_{\R^n\times\R^n}\bR^{xy}(\cdot)\hat{\pi}(dxdy)
            \end{equation}
        is the solution to \eqref{Definition:DynamicSBP}, which means that $\hat{\bP}_{0T}=\hat{\pi}$ and $\hat{\bP}$ shares its bridges with $\bR$ in the sense that
            \[
                \hat{\bP}^{xy}=\bR^{xy} \text{ for $\hat{\pi}$-a.e. $(x,y)$}.
            \]
        \item Conversely, if $\hat{\bP}$ is the solution to \eqref{Definition:DynamicSBP}, then $\hat{\pi}:=\hat{\bP}_{0T}$ is the solution to \eqref{Definition:StaticSBP}.
        
    \end{enumerate} 
\end{thm}
\noindent The idea of the proof is simple but worth repeating. Uniqueness follows from the fact that KL divergence is strongly convex and non-negative. As for the relations between the two solutions, we first note that, for any admissible $\bP$, the additive property of KL divergence \cite[Appendix A]{leonard2013survey} yields 
\begin{equation*}
    \KL{\bP}{\bR}=\KL{\bP_{0T}}{\bR_{0T}}+\int_{\R^n\times\R^n}\KL{\bP^{xy}}{\bR^{xy}}\bP_{0T}(dxdy).
\end{equation*}
Combining with \eqref{Equation:Disintegration}, it is easy to see that, if $\hat\pi$ is a solution to \eqref{Definition:StaticSBP}, then $ \KL{\bP}{\bR}$ achieves its minimum when $\bP_{0T}=\hat{\pi}$ and $\KL{\bP^{xy}}{\bR^{xy}}=0$ for $\bP_{0T}$-a.e. $(x,y)$, the latter of which implies that $\hat{\bP}^{xy}=\bR^{xy}$ for $\bP_{0T}$-a.e. $(x,y)$, and hence we arrive at \eqref{StaticToDynamic}. The second statement follows from a similar argument. \\

The static SBP is often simpler to study. Once a static Schr\"odinger bridge $\hat\pi$ is found, a solution $\hat\bP$ to the dynamic SBP becomes accessible following the theorem above (see \cite{gaussianbridge} for a recent example). The existence and the structure of $\hat\pi$ are thoroughly studied in the SBP literature \cite{fortet1940resolution,jamison1974reciprocal,jamison1975markov,follmer1988random,leonard2013survey,leonard2014some,follmer1997entropy,wakolbinger1990schrodinger}. In the discussion that follows, we give an overview of some theoretical results that will relevant to our study of the jump-diffusion SBP.

To begin, we recall that $\Pi(\rho_0,\rho_T)$ is compact in $\cP(\R^n\times\R^n)$ and that $\KL{\cdot}{\bR_{0T}}$ is lower-bounded and lower semi-continuous. Hence, the static SBP \eqref{Definition:StaticSBP} with $R=\bR_{0T}$ admits a unique solution $\hat\pi$ if and only if there exists some $\pi\in\Pi(\rho_0, \rho_T)$ such that $\KL{\pi}{\bR_{0T}}<\infty$\footnote{We also note that a necessary condition for $\hat\pi$ to exist is $\KL{\rho_0}{\bR_0}<\infty$ and $\KL{\rho_T}{\bR_T}<\infty$.}. This condition can be verified by different means, for example by checking the independent coupling $\pi=\rho_0\otimes\rho_T$ as a test case. Once the existence of $\hat\pi$ is established, the next step concerns the Radon-Nikodym derivative $d\hat\pi/d\bR_{0T}$. It is known that in many situations, this derivative is in fact a product of functions in separate coordinates. 
\begin{thm}\label{Theorem:fg-existence}\cite{leonard2013survey}
Assume that $\bR,\rho_0,\rho_T$ are such that the static SBP \eqref{Definition:StaticSBP} with reference measure $R=\bR_{0T}$ and target marginals $\rho_0,\rho_T$ admits a unique solution $\hat\pi$. If $\bR$ is Markov and satisfies $\bR_{0T}\ll \bR_0\otimes \bR_T$, then there exist non-negative measurable functions $\mathfrak{f}, \mathfrak{g} : \R^n \to \R$ such that\\
1. $\hat\pi$ takes the following form:
\begin{equation}\label{Equation:pihat}
    \hat{\pi} = \mathfrak{f}(X_0)\mathfrak{g}(X_T)\bR_{0T};
\end{equation}\\
2. the unique solution to the dynamic SBP \eqref{Definition:DynamicSBP} with reference path measure $\bR$ and target endpoint distributions $\rho_0,\rho_T$ is given by
\begin{equation}\label{Equation:bP=fgbR}
    \hat{\bP}=\f(X_0)\g(X_T)\bR;
\end{equation}\\
3. the pair $(\mathfrak{f},\mathfrak{g})$ satisfies the \textup{Schr\"odinger system}
\begin{equation}\label{Equation:fgSchrodingerSystem}
    \begin{cases}
    \vspace{0.2cm}\displaystyle \mathfrak{f}(x)\Esub{\bR}{\mathfrak{g}(X_T)|X_0=x} = \frac{d\rho_0}{d\bR_0}(x)& \text{ for $\bR_0$-a.e. $x$},\\
    \displaystyle \mathfrak{g}(y)\Esub{\bR}{\mathfrak{f}(X_0)|X_T=y} = \frac{d\rho_T}{d\bR_T}(y)& \text{ for $\bR_T$-a.e. $y$}.
    \end{cases}
\end{equation}
\end{thm}
\noindent \textbf{Remark.} The condition $\bR_{0T}\ll \bR_0\otimes \bR_T$ is automatically satisfied if $\bR_{0T}$ is absolutely continuous with respect to a $\sigma$-finite product measure on $\R^n\times\R^n$ (e.g. the Lebesgue measure). 
Further, note that if $\cC$ denotes the set of all pairs $(\rho_0,\rho_T)$ for which \eqref{Definition:StaticSBP} admits a solution, then $\cC$ is convex. When  $(\rho_0,\rho_T)\in\cC$ is \textit{internal} in the sense that there exist $(\alpha_0,\alpha_T),(\beta_0,\beta_T)\in \cC$ and $c\in(0,1)$ such that $(\rho_0,\rho_T)=(1-c)(\alpha_0,\alpha_T)+c(\beta_0,\beta_T)$, then the functions $\f,\g$ are positive, which implies that $\hat\bP$ and $\bR$ are equivalent measures. Although we only assume $\f,\g$ are non-negative while conducting this work, if they are positive, then many of the calculations and expressions down the line may be simplified.\\

Throughout this work, we assume that the conclusion of Theorem \ref{Theorem:fg-existence} holds\footnote{We refer readers to \cite{leonard2013survey} for a detailed survey on conditions that support the conclusion of Theorem \ref{Theorem:fg-existence}.}, meaning that $\hat\bP$ exists and takes the representation \eqref{Equation:bP=fgbR}, where $\f,\g$ are non-negative measurable functions solving the system \eqref{Equation:fgSchrodingerSystem}. Readers familiar with the subject will recognize the product-shaped density $\frac{d\hat\bP}{d\bR}=\f(X_0)\g(X_T)$ as a generalized, symmetric \textit{$h$-transform} \cite{doob1984classical}. Indeed, by setting $\h(t,x):=\Esub{\bR}{\g(X_T)|X_t=x}$ for $(t,x)\in[0,T]\times\R^n$ and assuming $\h(0,\cdot)$ is positive, the equation \eqref{Equation:bP=fgbR} becomes
\begin{equation}\label{Equation:bP=h1/h0bR}
\hat\bP=\frac{\h(T,X_T)}{\h(0,X_0)}\bP,\text{ where }\bP:=\frac{d\rho_0}{d\bR_0}(X_0)\bR.
\end{equation}
That is, $\hat\bP$ can be viewed as an $h$-transform of $\bP$. Obviously, $\bP$ is the path measure obtained from $\bR$ by replacing the initial endpoint distribution $\bR_0$ with $\rho_0$.
We will discuss in detail the notion of $h$-transform for jump diffusions in Section \ref{Section:hTransform}, as it is an important tool in the study of SBP. For a comprehensive exposition on $h$-transforms, we refer readers to \cite[Chapter 11]{chung2005markov}.

\subsection{Diffusion SBP and Dynamic Schr\"odinger System}  
Since Schr\"odinger's thought experiment was originally based on the diffusion of hot gaseous particles, it is not surprising that the most extensively studied SBP is the diffusion case, that is, when $\bR$ is the law of an $\R^n$-valued diffusion process $\set{X_t}_{[0,T]}$ governed by the SDE
\begin{equation}\label{ReferenceItoDiffusion}
dX_t=b(t,X_t)dt+\sigma(t,X_t)dB_t; \,\, 0\leq t \leq T,
\end{equation}
where $\set{B_t}_{t\geq 0}$ is a standard $m$-dimensional Brownian motion \cite{jamison1975markov, nagasawa1989transformations,dai1991stochastic,follmer1988random}. In this section, we will review some results on the diffusion SBP that will be extended to the jump-diffusion setting in Section \ref{Section:smooth-coefficients}. In particular, we will adopt the setup from \cite{jamison1975markov}. Throughout this section, we will assume that the endpoint distributions $\bR_0,\bR_T$ admit probability density functions (and hence so do $\rho_0,\rho_T$) and, with slight abuse of notation, we will write these distributions as $\bR_0(x)dx$, etc.; we also assume that $\bR_0$, $b$ and $\sigma$ satisfy proper conditions\footnote{Typically, $\bR_0$ has finite second moment, $b,\sigma$ are bounded, continuous and satisfy a H\"older condition in the spatial variable, and $a:=\sigma\sigma^T$ satisfies a uniform ellipticity condition \cite{menozzi2021density}.} 
such that the SDE \eqref{ReferenceItoDiffusion} has a (weak) solution $\set{X_t}_{[0,T]}$ that is strong Markovian and admits an everywhere positive, continuous transition density $q(s,x,t,y)$ for $0\leq s< t\leq T$ and $x,y\in\R^n$.

Let $(\f,\g)$ be the solution of the Schr\"odinger system \eqref{Equation:fgSchrodingerSystem}. By defining the functions $\hat\varphi_0:=\f\cdot\bR_0$, $\varphi_0:=\rho_0/\hat\varphi_0$, $\varphi_T:=\g$ and $\hat\varphi_T:=\rho_T/\varphi_T$, we obtain an equivalent expression of the system \eqref{Equation:fgSchrodingerSystem} as a pair of integral equations coupled with nonlinear boundary constraints: for $\bR_{0T}$-a.e. $(x,y)\in\R^{2n}$,
\begin{equation}\label{phiphihat}
    \begin{cases}
     \varphi_0(x) = \int_{\R^n} q(0, x, T, y)\varphi_T(y)dy, \\
     \hat{\varphi}_T(y) = \int_{\R^n} q(0, x, T, y)\hat{\varphi}_0(x)dx,\\ 
     \rho_0(x)=\varphi_0(x)\hat{\varphi}_0(x),\\
     \rho_T(y)=\varphi_T(y)\hat{\varphi}_T(y).
\end{cases}
\end{equation}
The development of the existence of a solution to such a system goes back to Fortet \cite{fortet1940resolution}, Beurling \cite{beurling1960automorphism} and Jamison \cite{jamison1974reciprocal,jamison1975markov}, and predates the modern formulation of the SBP.
\begin{thm}[Fortet-Beurling-Jamison existence theorem]\label{Theorem:FBJ-Existence}
    Under the setting above, there exists a unique (up to scaling) 4-tuple of functions $(\varphi_0,\varphi_T,\hat\varphi_0,\hat\varphi_T)$ solving the Schr\"odinger system \eqref{phiphihat}.
\end{thm}
Although later contributions and proofs are analytic in nature, Fortet's approach was to develop an iterative scheme based on the four equations in (\ref{phiphihat}). His algorithm in \cite{fortet1940resolution}, as well as some modern revisits of the algorithm \cite{essid2019traversing, chen2016entropic}, is based on a fixed-point method and belongs to what is now called \textit{Iterative Proportional Fitting} or \textit{Sinkhorn-type} \cite{sinkhorn1967concerning} algorithms (depending on the setting). This algorithm provides an efficient way of computing the Schr\"odinger system solution in applications and it has been widely used in optimal transport in recent years. In fact, the entropy-regularized optimal transport problem, which is equivalent to the static SBP \eqref{Definition:StaticSBP} with a specific choice of reference measure $R$, can be solved via this algorithm in a computationally efficient way \cite{cuturi2013sinkhorn}.\\

We now switch to a ``dynamic'' view. Given the solution $(\varphi_0,\varphi_T,\hat\varphi_0,\hat\varphi_T)$ to the Schr\"odinger system \eqref{phiphihat}, we observe that if $\varphi,\hat\varphi$ are two functions on $[0,T]\times\R^n $ defined as
\begin{equation*}
    \begin{cases}
     \varphi(t,x) := \int_{\R^n} q(t, x, T, y)\varphi_T(y)dy, \\
     \hat{\varphi}(t,x) := \int_{\R^n} q(0, y, t, x)\hat{\varphi}_0(y)dy,
\end{cases}\quad\text{for}(t,x)\in[0,T]\times\R^n,
\end{equation*}
then $\varphi,\hat\varphi$ satisfy the boundary constraints
\[
\varphi(0,\cdot)=\varphi_0, \;  \hat\varphi(0,\cdot)=\hat\varphi_0, \; 
\varphi(T,\cdot)=\varphi_T, \; \hat\varphi(T,\cdot)=\hat\varphi_T.
\]
In other words, we can rewrite the solution to \eqref{phiphihat} as $(\varphi(0,\cdot),\hat\varphi(0,\cdot),\varphi(T,\cdot),\hat\varphi(T,\cdot))$.
If, in addition, $\varphi,\hat\varphi$ are functions of class $C^{1,2}$, then by invoking the Kolmogorov backward and forward equations for the SDE \eqref{ReferenceItoDiffusion}, we obtain from \eqref{phiphihat} the \textit{dynamic Schr\"odinger system}
\begin{equation}\label{Equation:ItoSchrodingerSystem}
    \begin{dcases}
    \frac{\partial \varphi}{\partial t}(t,x)+(b\cdot\nabla \varphi)(t,x) +\frac{1}{2} \sum_{i,j}a_{ij}\frac{\partial^2\varphi}{\partial x_i \partial x_j}\,(t,x)=0\text{ for }(t,x)\in(0,T)\times\R^n,\\
    \frac{\partial \hat{\varphi}}{\partial t}(t,x)+ \nabla\cdot(b\hat\varphi)(t,x) -\frac{1}{2} \sum_{i,j}\frac{\partial^2(a_{ij}\hat{\varphi})}{\partial x_i \partial x_j}(t,x)=0\text{ for }(t,x)\in(0,T)\times\R^n,\\
    \rho_0(x)=\varphi(0,x)\hat{\varphi}(0,x)\text{ for }x\in\R^n,\\
    \rho_T(x)=\varphi(T,x)\hat{\varphi}(T,x)\text{ for }x\in\R^n.
    \end{dcases}
\end{equation} 
Recall from \eqref{Equation:bP=fgbR} that the (dynamic) Schr\"odinger bridge is given by $\hat\bP=\f(X_0)\g(X_T)\bR$, which can be reparametrized in terms of $\varphi,\hat\varphi$ as 
\[
\hat\bP=\frac{\varphi(T,X_T)}{\varphi(0,X_0)}\frac{\rho_0(X_0)}{\bR_0(X_0)}\bR.
\]
As previously noted, this expression corresponds to an $h$-transform (by the harmonic function $\varphi$). We summarize in the theorem below the main results on $\hat\bP$ in the current setting of diffusion SBPs, part of which were first presented by Jamison in \cite{jamison1975markov}.
\begin{thm}\label{thm:DiffusionSBP} Under the setting above, there exists a unique 
(up to scaling) pair of non-negative functions $(\varphi,\hat\varphi)$ on $[0,T]\times\R^n$ such that $(\varphi(0,\cdot),\hat\varphi(0,\cdot),\varphi(T,\cdot),\hat\varphi(T,\cdot))$ solves the Schr\"odinger system \eqref{phiphihat}, and the marginal distribution of $\hat\bP$ at $t\in[0,T]$ admits a probability density function given by
\begin{equation*}
    \hat\bP_t(x)=\varphi(t,x)\hat\varphi(t,x)\text{ for every }x\in\R^n.
\end{equation*}
If further $\varphi,\hat\varphi$ are of class $C^{1,2}$, then they satisfy the dynamic Schr\"odinger system \eqref{Equation:ItoSchrodingerSystem}, and the canonical process under $\hat{\bP}$ solves the modified SDE
\begin{equation*}
dX_t=\pran{b+\sigma\sigma^T\nabla\log \varphi}(t,X_t)dt+\sigma(t,X_t)dB_t;\, t\in[0,T]. 
\end{equation*}  
\end{thm}
In Section \ref{Section:smooth-coefficients}, we extend these results to jump-diffusion SBPs. The condition of having a positive \textit{and} continuous transition density in Theorem \ref{Theorem:FBJ-Existence} and Theorem \ref{thm:DiffusionSBP} is satisfied by a large class of diffusions. However, results of this type are not readily available for general jump diffusions. In this work we will resort to Theorem \ref{Theorem:fg-existence} as our starting point instead.
\section{$h$-Transform of Jump Diffusions}\label{Section:hTransform}
As our goal is to study the SBP for jump diffusions, we will now return to the setup in the introduction. From this point on we will assume that $\Omega$ is the Skorokhod space $D([0,T];\R^n)$, each sample path $\omega\in\Omega$ is a c\`adl\`ag function from $[0,T]$ to $\R^n$, and the corresponding Borel $\sigma$-algebra $\cB_{\Omega}$ is generated by the time projections $\set{X_t}_{[0,T]}$. We take $\set{\cF_t}_{[0,T]}$ to be the natural filtration of $\set{X_t}_{[0,T]}$ (or its usual augmentation when necessary). 
To understand the dynamics of the jump-diffusion Schr\"odinger bridge, one must investigate properties of the $h$-transform in the jump-diffusion setting. 

We will first review preliminary notions on $h$-transforms, and then study the $h$-transform of a jump diffusion  through (1) the generator of the jump diffusion and (2) the SDE corresponding to the jump diffusion. Following each approach, we develop a theory of $h$-transforms that extends the existing literature from the diffusion case ($\gamma = 0$) to the jump-diffusion setting. The ``jumping behavior'' of our model allows us to study a broader class of dynamics,  but it also introduces an additional layer of complexity to the problem. In the generator approach, the jump component turns a local operator into a non-local one\footnote{Sometimes, the jump-diffusion operator considered in this work is also referred to as a local-nonlocal mixed-type operator in the literature.}, and in the SDE approach, the jump terms appear as new stochastic integrals. In the following subsections, we will devise proper techniques to tackle these changes. 
\subsection{$h$-Transform: Preliminaries}
The $h$-transform was first introduced by Doob in his work on potential theory \cite{doob1957conditional,doob1984classical}. A common situation in which the $h$-transform arises is when a Markov process is conditioned to exit a ball $B$ at a given point $x_0\in \partial B$. In its most general sense, the $h$-transform is a combination of conditioning and killing, but we will not consider the latter. Let $L$ be the jump-diffusion operator given by \eqref{ReferenceGenerator} with coefficients satisfying the Assumptions \ref{Assumptions B}. 
Throughout this subsection we will assume that \ref{Assumption:MartingaleProblem-s,x} holds, and we denote by $\bP$ the solution to the martingale problem for $(L,\rho_0)$. The coordinate process $\set{X_t}_{[0,T]}$ under $\bP$ is a strong Markov process and for any $0 \leq s \leq t \leq T$, $x\in\R^n$, we denote by $P_{s,t}(x,\cdot)$ its transition distribution.
Throughout this subsection, we also assume that $h:[0,T]\times\R^n\to[0,\infty)$ is a non-negative and measurable function, and set 
\begin{equation}\label{definition_A_h}
\cA_h:=\set{(t,x)\in [0,T]\times\R^n:h(t,x)>0}.
\end{equation}
We further assume that
\begin{equation}\label{definition_A0r0}
\text{if }A_0:=\set{x\in \R^n:h(0,x)>0}\text{, then }r_0:=\bP(\set{X_0\in A_0})>0.
\end{equation}
\begin{dfn}\label{def:meanvalueproperty}
We say that $h$ satisfies the \textit{mean-value property} if for all $0 \leq s \leq t \leq T$ and $x\in\R^n$,
        \begin{equation}\label{Equation:MeanValueProperty}
            h(s,x)=\int_{\R^n} h(t,y)P_{s,t}(x,dy).
        \end{equation} 
\end{dfn}
\begin{lem}\label{lem_martingale1_h(T)/h(0)}
    Suppose that $h$ satisfies the mean-value property \eqref{Equation:MeanValueProperty}. Then 
    \[
    \set{\one_{\set{X_0\in A_0}}\frac{h(t,X_t)}{h(0,X_0)} : 0 \leq t \leq T}
    \]
    is an $\set{\cF_t}$-martingale under $\bP$.
\end{lem}
\begin{proof}
First note that the mean-value property \eqref{Equation:MeanValueProperty} guarantees that, for every  $t\in[0,T]$, $\one_{\set{X_0\in A_0}}\frac{h(t,X_t)}{h(0,X_0)}$ is integrable under $\bP$, because 
\begin{equation*}
    \bE_{\bP}\brac{\one_{\set{X_0\in A_0}}\frac{h(t,X_t)}{h(0,X_0)}}=\int_{A_0}\frac{\int_{\R^n}h(t,y)P_{0,t}(x,dy)}{h(0,x)}\rho_0(dx)=\rho_0(A_0)=r_0. 
\end{equation*}
Similarly, for every $t\in[0,T]$ and arbitrary $A\in\cF_t$, 
\begin{align*}
    \bE_{\bP}\brac{\one_{\set{X_0\in A_0}}\frac{h(T,X_T)}{h(0,X_0)}\one_A}&=\bE_{\bP}\brac{\one_{\set{X_0\in A_0}}\frac{\bE_{\bP^{t,X_t}}\brac{h(T,X_T)}}{h(0,X_0)}\one_A} \\
        &=\bE_{\bP}\brac{\one_{\set{X_0\in A_0}}\frac{h(t,X_t)}{h(0,X_0)}\one_A} \text{ by \eqref{Equation:MeanValueProperty}}. 
\end{align*}
Thus, $\displaystyle \bE_{\bP}\brac{\one_{\set{X_0\in A_0}}\frac{h(T,X_T)}{h(0,X_0)}\bigg|\cF_t}=\one_{\set{X_0\in  A_0}}\frac{h(t,X_t)}{h(0,X_0)}$, and the claim follows.
\end{proof}
\begin{dfn} 
As a result of Lemma \ref{lem_martingale1_h(T)/h(0)},  
\begin{equation}\label{Definition:h-transform,general case with r0}
    \bP^h:=\frac{\one_{\set{X_0\in A_0}}}{r_0}\frac{h(T,X_T)}{h(0,X_0)}\bP
\end{equation} 
defines a probability measure on $\Omega$, and $\bP^h$ is called the \textup{$h$-transform} of $\bP$. 
\end{dfn}
Let us now examine the transition distribution of $\set{X_t}_{[0,T]}$ under the path measure $\bP^h$.
\begin{lem}\label{lem_TranDistHtransform}
    Suppose that $h$ satisfies the mean-value property \eqref{Equation:MeanValueProperty} and $\bP^h$ is as in \eqref{Definition:h-transform,general case with r0}. Then, under $\bP^h$, the process $\set{X_t}_{[0,T]}$ admits the transition distribution
    \begin{equation}
        P^h_{s,t}(x,dy)=
        \begin{cases}\label{hTransformPh}
            \displaystyle \frac{h(t,y)}{h(s,x)}P_{s,t}(x,dy) & \text{ if }(s,x)\in \cA_h, \\
            0 & \text{ otherwise}.
        \end{cases}
    \end{equation}
\end{lem}
\begin{proof}
    First, observe that by the mean-value property \eqref{Equation:MeanValueProperty}, if $h(s,x)>0$, then $P^h_{s,t}(x,\cdot)$ is a probability measure on $\R^n$. Next, it suffices to check the finite-dimensional distributions of $\set{X_t}_{[0,T]}$ under $\bP^h$. Let $0 =t_0 < t_1 <\dots<t_N\leq T$ and $B_0,\dots,B_N\in\cB(\R^n)$ be arbitrary and define $\tilde{B_i}=B_i \cap \set{x_i: (t_i,x_i)\in \cA_h}$ for $i=0,1,\dots,N$. Then
    \begin{align*}
        &\bP^h\pran{X_0\in B_0, X_{t_1}\in B_1,\dots, X_{t_N}\in B_N} \\
            =& \bE_{\bP}\brac{\one_{\set{X_0\in A_0}}\frac{h(T,X_T)}{h(0,X_0)};X_{t_i} \in B_i, i=0,1,\dots,N}\\
            =& \int_{\tilde{B}_0}\int_{B_1}\dots\int_{B_N}\int_{\R^n}\frac{h(T,y)}{h(0,x_0)}P_{t_N,T}(x_{N},dy)P_{t_{N-1},t_{N}}(x_{N-1},dx_{N})\dots P_{0,t_1}(x_{0},dx_1)\rho_0(dx_0)\\
            \overset{(\dagger)}{=}& \int_{\tilde{B}_0}\int_{\tilde{B}_1}\dots\int_{\tilde{B}_N}\int_{\R^n}\frac{h(t_1,x_1)}{h(0,x_0)}\frac{h(t_2,x_2)}{h(t_1,x_1)}\dots\frac{h(T,y)}{h(t_N,x_N)}P_{t_N,T}(x_{N},dy)\dots P_{0,t_1}(x_{0},dx_1)\rho_0(dx_0)\\
            =&\int_{B_0}\int_{B_1}\dots\int_{B_N}\int_{\R^n}P^h_{t_N,T}(x_{N},dy)P^h_{t_{N-1},t_{N}}(x_{N-1},dx_{N})\dots P^h_{0,t_1}(x_{0},dx_1)\rho_0(dx_0).
    \end{align*}
    To see the equality in $(\dagger)$, suppose that for some $1 \leq i \leq N$, $h(t_i,x_i)=0$ for some $x_i\in B_i$. Then using \eqref{Equation:MeanValueProperty},
    \begin{align*}
        \int_{B_{i+1}}\dots\int_{\R^n}h(T,y)P_{t_N,T}(x_{N},dy)&\dots P_{t_{i},t_{i+1}}(x_{i},dx_{i+1})\\
        &\leq \int_{\R^n}\dots\int_{\R^n}h(T,y)P_{t_N,T}(x_{N},dy)\dots P_{t_{i},t_{i+1}}(x_{i},dx_{i+1})\\
        &= h(t_i,x_i)=0.
    \end{align*}
    Thus, the integral
    \begin{equation*}
        \int_{B_i\cap \set{x_i: (t_i,x_i)\notin \cA_h}} \int_{B_{i+1}}\dots\int_{\R^n}h(T,y)P_{t_N,T}(x_{N},dy)\dots P_{t_{i},t_{i+1}}(x_{i},dx_{i+1})P_{t_{i-1},t_i}(x_{i-1},dx_i)
    \end{equation*}
    is equal to $0$, and the equality $(\dagger)$ follows.
\end{proof}

The proof of Lemma \ref{lem_TranDistHtransform} shows that the points outside of the set $\cA_h$ have no effect on the the transition distributions under $\bP^h$. In fact, $\bP^h$ is supported on the set of sample paths that stay entirely in $\cA_h$. To make this statement accurate, we consider the following hitting times:
    \begin{equation}\label{def_tau_n}
    \text{for every $s\in[0,T]$ and every $n\geq1$},\tau^s_n:=\inf\set{t\geq s:h(t,X_t)<\frac{1}{n}}.
    \end{equation}
    Then, $\tau^s_n$'s are increasing $\set{\cF_t}$-stopping times\footnote{In general, $\tau^s_n$'s are optional times and hence stopping times when $\set{\cF_t}$ is taken to be its usual augmentation.}, and so is 
    \begin{equation}\label{def_tau}
    \tau^s:=\sup_{n\geq 1}\tau^s_n = \lim_{n\rightarrow\infty}\tau^s _n.
    \end{equation}
We are now ready to give an equivalent representation of $\bP^h$.
\begin{prop}\label{prop:equiv represent of bP^h}
Assume that $h$ satisfies the mean-value property \eqref{Equation:MeanValueProperty}. Let $\bP^h$ be as defined in \eqref{Definition:h-transform,general case with r0}. Then,
\begin{equation}\label{Definition:h-transform,general case with tau0}
    \bP^h=\frac{\one_{\set{\tau^0>T}}}{r_0}\frac{h(T,X_T)}{h(0,X_0)}\bP.
\end{equation}
In other words, $\bP^h$ is supported on $\set{\tau^0>T}$, where $(t,X_t)\in\cA_h$ for all $t\in[0,T]$.
\end{prop}   
\begin{proof}
First, we claim that for every $0\leq s\leq t\leq T$ and every $x\in\R^n$, \begin{equation}\label{h(t,X_t)_vanish_post_tau^s}
    \bE_{\bP^{s,x}}[h(t,X_t)\one_{\set{\tau^s\leq t}}]=0.
    \end{equation}
    Indeed, since $\one_{\set{\tau^s_n\leq t}}\rightarrow \one_{\set{\tau^s\leq t}}$ as $n\rightarrow\infty$, using Fatou's lemma and the strong Markov property of $\set{X_t}_{[0,T]}$,
    \begin{equation*}
    \begin{aligned}
    \bE_{\bP^{s,x}}[h(t,X_t)\one_{\set{\tau^s\leq t}}] & \leq \liminf_{n\rightarrow\infty}\bE_{\bP^{s,x}}[h(t,X_t)\one_{\set{\tau^s_n\leq t}}]\\
    & = \liminf_{n\rightarrow\infty}\bE_{\bP^{s,x}}\brac{\bE_{\bP^{\tau^s_n,X_{\tau^s_n}}}[h(t,X_t)]\one_{\set{\tau^s_n\leq t}}}\\
    & = \liminf_{n\rightarrow\infty}\bE_{\bP^{s,x}}[h(\tau^s_n,X_{\tau^s_n})\one_{\set{\tau^s_n\leq t}}]\text{ by the mean-value property of $h$}\\
    & \leq \liminf_{n\rightarrow\infty}\frac{1}{n}\bP(\tau^s_n\leq t)=0 \text{ by right-continuity of $\set{X_t}_{[0,T]}$}.
    \end{aligned}
    \end{equation*}
It follows immediately from \eqref{h(t,X_t)_vanish_post_tau^s} that 
    \begin{equation*}
    \bP^h(\tau^0\leq T)=\frac{1}{r_0}\int_{A_0}\frac{\bE_{\bP^{0,x}}[h(T,X_T)\one_{\tau^0\leq T}]}{h(0,x)}\rho_0(dx)=0,
    \end{equation*}
and hence (\ref{Definition:h-transform,general case with tau0}) holds.
\end{proof}
In later sections, in our study of the SBP, we will consider $h$ in the form of $h(t,x)=\int g(y)P_{t,T}(x,dy)$, $(t,x)\in[0,T]\times\R^n$, for some non-negative $g$. The mean-value property is always satisfied by such an $h$, and hence $\bP^h$ is always well defined. We now have two equivalent representations \eqref{Definition:h-transform,general case with r0} and \eqref{Definition:h-transform,general case with tau0} for $\bP^h$, both of which will be useful in later sections.\\

We close this subsection by introducing a notion which is closely related to the mean-value property andn which plays an important role in the study of the $h$-transform when $h$ possesses a certain level of regularity. 

\begin{dfn}\label{Definition:HarmonicFunction}
    We say that a function $h:[0,T]\times\R^n \to [0,\infty)$ is \textup{harmonic} if $h$ is of class  $C^{1,2}$ on $[0,T]\times\R^n$, and $h$ satisfies $\cL h=\frac{\partial h}{\partial t}+Lh=0$ on $(0,T)\times\R^n$.
\end{dfn}

\noindent \textbf{Remark.} It is straightforward to see that, under Assumption \ref{Assumption:MartingaleProblem-s,x}, when $h$ is sufficiently regular, e.g., $h\in C^{1,2}_c$, the notion of $h$ being harmonic is equivalent to $h$ satisfying the mean-value property. However, our definition of the $h$-transform only requires the latter.

\subsection{$h$-Transform: the Generator Approach}\label{Subsection:Generator Approach}
In this subsection we study the $h$-transform of a path measure by analyzing the change occurred in the generator. We further derive the martingale problem formulation corresponding to the transformed path measure in the case when $h$ is a harmonic function.

We assume that the operator $L$ in (\ref{ReferenceGenerator}) satisfies Assumption \ref{Assumption:MartingaleProblem-s,x} and $\bP$ is a solution to the martingale problem for $(L,\rho_0)$. For a non-negative function $h$ of class $C^{1,2}$ on $[0,T]\times\R^n$, we define $\cA_h$ as in \eqref{definition_A_h} and $\cL^h,L^h$ integro-differential operators as: for $f\in C^{1,2}(\R^+\times \R^n)$,
\begin{equation}\label{Definition:L^h}
    \cL^hf:=\one_{\cA_h}\frac{\cL(hf)}{h} \text{ and } L^hf:=\one_{\cA_h}(\cL^hf-\frac{\partial f}{\partial t}).
\end{equation}
As discussed in the previous subsection, when $h$ satisfies the mean-value property \eqref{Equation:MeanValueProperty}, we refer to the path measure $\bP^h$ defined in \eqref{Definition:h-transform,general case with r0} as the $h$-transform of $\bP$. We will show that if, in addition, $h$ is harmonic, then $\bP^h$ solves the martingale problem for $L^h$ with a proper initial distribution. First, let us present the explicit expression of $L^h$.
\begin{lem}\label{Theorem:ProductRule}
Let $L$ be the jump-diffusion operator defined in \eqref{ReferenceGenerator}, $h$ be a non-negative function of class $C^{1,2}$ on $[0,T]\times\R^n$, and $\cA_h$ be as in \eqref{definition_A_h}. Define, for every $(t,x)\in \cA_h$, 
\begin{equation}
\label{definition_b^h}\begin{aligned} b^h(t,x):=b(t,x)& +\sigma\sigma^T\nabla
 \log h(t,x)+ \\
 & \int_{|z|\leq 1} \frac{h(t,x+\gamma(t,x,z))-h(t,x)}{h(t,x)}\gamma(t,x,z)\nu(dz),
\end{aligned}
\end{equation}
and
\begin{equation}\label{definition_nu^h}
 \nu^h(t,x;dz) :=  \frac{h(t,x+\gamma(t,x,z))}{h(t,x)}\nu(dz).
\end{equation}
Then for every $f\in C^{1,2}(\R^+\times \R^n)$, at $(t,x)\in \cA_h$, \begin{equation*}
    L(hf)(t,x)=\left (h\tilde{L}f+fLh\right ) (t,x),
\end{equation*}
where $\tilde{L}$ is the operator defined in \eqref{ReferenceGenerator} with $b$ replaced by $b^h$, and $\nu$ replaced by $\nu^h$. 
Equivalently, writing $\tilde{\cL}:=\frac{\partial}{\partial t}+\tilde{L}$, at $(t,x)\in \cA_h$, 
\begin{equation}\label{Equation:GeneratorProduct-withd/dt}
    \cL(hf) (t,x) =\left(h\tilde{\cL}f+f\cL h\right)(t,x).
\end{equation}
In particular, if $h$ is harmonic, then $$\tilde{\cL}f=\frac{1}{h}\cL(hf)=\cL^hf\text{ and }L^hf=\tilde{L}f\text{ on }\cA_h.$$
\end{lem}
\noindent All the equations in the lemma are derived by direct computations, and hence the proof is omitted. We will only point out that, under Assumptions \ref{Assumptions B} on $\sigma,b,\gamma$, it is easy to verify that for every $(t,x)\in \cA_h$, the integral in (\ref{definition_b^h}) is finite and for $\nu^h(t,x,\cdot)$ defined in (\ref{definition_nu^h}),
\begin{equation*}
\int_{\R^\ell}(1\wedge|z|^2)\nu^h(t,x;dz)<\infty,
\end{equation*}
and hence $b^h$ is well defined and $\nu^h(t,x;\cdot)$ is a L\'evy measure.\\

We are now ready to formulate the martingale problem statement involving $\bP^h$ and $L^h$.
\begin{thm}\label{Theorem:MartingaleProblem-Lh}
    Assume that Assumption \textup{\ref{Assumption:MartingaleProblem-s,x}} holds for the jump-diffusion operator $L$ and $\bP$ is the solution to the martingale problem for $(L,\rho_0)$. Let $h$ be a harmonic function that satisfies the mean-value property, and let $\bP^h$ and $L^h$ be defined as in \eqref{Definition:h-transform,general case with r0} and \eqref{Definition:L^h} respectively. Then $\bP^h$ solves the martingale problem for $(L^h,\rho^h_0)$, where $\rho^h_0:=\frac{\one_{A_0}}{r_0}\rho_0$ with $A_0$ and $r_0$ defined as in \eqref{definition_A0r0}. Moreover, the canonical process $\set{X_t}_{[0,T]}$ under $\bP^h$ is a strong Markov process.
\end{thm}
\begin{proof}
    We first check the initial distribution. Let $B\subseteq \R^n$ be any Borel set, then
    \begin{align*}
        \bP^h(X_0\in B) &= \bE_{\bP}\brac{\frac{1}{r_0}\frac{h(T,X_T)}{h(0,X_0)}; X_0\in (B\cap A_0)} \\
        &= \int_{B\cap A_0}\frac{1}{r_0}\frac{1}{h(0,x)}\int_{\R^n}h(T,y)P_{0,T}(x,dy)\rho_0(dx)\\
        &=\frac{1}{r_0}\rho_0(B\cap A_0)\text{ by the mean-value property of $h$.}
    \end{align*}
    
    Now, we verify that for every $f\in C^{1,2}_c$, $\set{f(t,X_t)-\int_0^t\cL^hf(r,X_r)dr \,:\,0\leq t \leq T}$ under $\bP^h$
    is a martingale. Let $0\leq s \leq t \leq T$ and $A\in \cF_s$ be arbitrary. Then, by Lemma \ref{lem_martingale1_h(T)/h(0)},
\begin{align*}
    \int_A (f(t,X_t)-f(s,X_s))d\bP^h
        &= \frac{1}{r_0} \int_{A}\one_{A_0}(X_0) \frac{h(T,X_T)}{h(0,X_0)}\pran{f(t,X_t)-f(s,X_s)}d\bP \\
        &= \frac{1}{r_0}\int_A \frac{\one_{A_0}(X_0)}{h(0,X_0)}\pran{(hf)(t,X_t)-(hf)(s,X_s)}d\bP \\
        &=\frac{1}{r_0}\int_A \frac{\one_{A_0}(X_0)}{h(0,X_0)}\pran{(hf)(t,X_t)-(hf)(t\wedge \tau^s,X_{t\wedge \tau^s})}d\bP\\
        &\hspace{0.8cm}+\frac{1}{r_0}\int_A \frac{\one_{A_0}(X_0)}{h(0,X_0)}\pran{(hf)(t\wedge \tau^s,X_{t\wedge \tau^s})-(hf)(s,X_s)}d\bP,
\end{align*}
where $\tau^s$ is as in \eqref{def_tau}. We first treat the integral involving the difference in $hf$ between $t\wedge\tau^s$ and $t$, and it suffices to restrict the integral to $\set{\tau^s\leq t}$. That is,
\begin{align*}
\frac{1}{r_0}\int_{A\cap\set{\tau^s\leq t}} \frac{\one_{A_0}(X_0)}{h(0,X_0)}&\left((hf)(t,X_t)-(hf)(\tau^s,X_{\tau^s})\right)d\bP\\
    &=\frac{1}{r_0}\int_{A\cap\set{\tau^s\leq t}} \frac{\one_{A_0}(X_0)}{h(0,X_0)}\left(h(t,X_t)(f(t,X_t)-f(\tau^s,X_{ \tau^s})\right)d\bP\\
& \hspace{0.5cm} +\frac{1}{r_0}\int_{A\cap\set{\tau^s\leq t}} \frac{\one_{A_0}(X_0)}{h(0,X_0)}\left((h(t,X_t)-h(\tau^s,X_{\tau^s})\right)f(\tau^s,X_{\tau^s})d\bP.
\end{align*}
Conditioning on $\cF_{\tau^s}$ inside the second integral above and using the strong Markov property of $\set{X_t}_{[0,T]}$ under $\bP$, we see that the second integral vanishes due to the mean-value property of $h$. Conditioning on $\cF_s$ inside the first integral and using the boundedness of $f$, we see that the first integral is bounded from above by
\begin{equation*}
\frac{\norm{f}_\text{u}}{r_0}\int_A \frac{\one_{A_0}(X_0)}{h(0,X_0)}\bE_{\bP^{s,X_s}}[h(t,X_t)\one_{\tau^s\leq t}]d\bP,
\end{equation*}
which also vanishes according to \eqref{h(t,X_t)_vanish_post_tau^s}. Therefore, we have
\begin{flalign*}
     \int_A &\left(f(t,X_t)-f(s,X_s)\right)d\bP^h\\
        &= \frac{1}{r_0}\int_A \frac{\one_{A_0}(X_0)}{h(0,X_0)}\pran{(hf)(t\wedge\tau^s,X_{t\wedge \tau^s})-(hf)(s,X_s)}d\bP\\
        &= \frac{1}{r_0}\int_A \frac{\one_{A_0}(X_0)}{h(0,X_0)}\int_s^{t\wedge\tau^s} \cL(hf)(r,X_r)\,dr\,d\bP \text{ by \ref{Assumption:MartingaleProblem-s,x}, since $hf \in C^{1,2}_c$} \\
        &=\frac{1}{r_0}\int_A \int_s^{t\wedge\tau^s} \one_{A_0}(X_0)\frac{h(r,X_r)}{h(0,X_0)} \cL^hf(r,X_r)\,dr\,d\bP \text{ since $\cL h=0$ and $\set{(r,X_r)}_{[s,\tau^s)}\subseteq \cA_h$}\\
        &=\frac{1}{r_0}\int_A \one_{A_0}(X_0)\frac{h(T,X_T)}{h(0,X_0)}\int_s^{t\wedge\tau^s} \cL^hf(r,X_r)\,dr\,d\bP \text{ using the martingale property again}\\
        &= \int_A \int_s^{t\wedge\tau^s} \cL^hf(s,X_s)\,ds\,d\bP^h=\int_A \int_s^t \cL^hf(s,X_s)\,ds\,d\bP^h \text{ as desired,}
\end{flalign*}
where the last equation is due to the fact that $\bP^h$ is supported on $\set{
\tau^0>T}\subseteq\set{\tau^s>t}$.

The strong Markov property of $\set{X_t}_{[0,T]}$ under $\bP^h$ can be established following a similar argument.
\end{proof}
We have identified $\bP^h$ as the solution to the martingale problem for $(L^h,\rho_0^h)$. Seeing that $L^h$ takes the same form as $L$ with $b$ replaced by $b^h$ and $\nu$ replaced by $\nu^h$, it is tempting to conclude that  $(\set{X_t}_{[0,T]},\bP^h)$ is a (weak) solution to the SDE
\begin{equation*}
\begin{aligned}
   dX_t=b^h(t,X_t)dt&+\sigma(t,X_t)dB_t \\
   &+ \int_{|z|\leq 1}\gamma(t,X_{t-},z)\tilde{N}^h(dt,dz) 
   + \int_{|z|>1}\gamma(t,X_{t-},z)N(dt,dz),
   \end{aligned}
\end{equation*}
where $\tilde{N}^h(dt,dz):=N(dt,dz)-\nu^h(dz)dt$.
However, it is not clear in the jump-diffusion case when the solution to the martingale problem leads to a weak solution of the corresponding SDE. In the next subsection, we will impose stronger assumptions and use the Girsanov theorem for jump diffusions to derive the SDE satisfied by $\set{X_t}_{[0,T]}$ under $\bP^h$.
\subsection{$h$-Transform: the SDE Approach}
In this subsection, we put our jump-diffusion model in the SDE setting and investigate the $h$-transform from the point of view of the Girsanov theorem. We assume that \ref{Assumption:SDEsolution} holds, that is, there exists a weak solution $(\set{X_t}_{[0,T]},\bP)$ to the SDE \eqref{ReferenceLevyIto} where $\bP$ is a path measure on $\Omega=D([0,T];\R^n)$ and $\set{X_t}_{[0,T]}$ under $\bP$ is strong Markov with initial distribution $X_0\sim\rho_0$. Recall that under Assumptions \ref{Assumptions B} imposed on $\sigma$, $b$ and $\gamma$, all the stochastic integrals involved in \eqref{ReferenceLevyIto} are well defined. Throughout this subsection, we consider $h$ to be a non-negative harmonic function that satisfies the mean-value property \eqref{Equation:MeanValueProperty}, and let $\bP^h$ be the $h$-transform of $\bP$ given by  \eqref{Definition:h-transform,general case with tau0}.
We first demonstrate a particular Girsanov transformation which coincides with $\bP^h$, and then we derive the SDE representation of the process under $\bP^h$. 
\begin{thm}\label{GirsanovHTransform}
    Assume that \textup{\ref{Assumption:SDEsolution}} holds for the SDE \eqref{ReferenceLevyIto} and $(\set{X_t}_{[0,T]},\bP)$ is the solution to \eqref{ReferenceLevyIto} that is a strong Markov process with $X_0\sim\rho_0$. Let $h$ be a non-negative harmonic function and satisfies the mean-value property \eqref{Equation:MeanValueProperty}. 
    Let $\set{Z_t}_{[0,T]}$ be the process given by 
    
    \begin{equation}\label{Equation:Girsanov Zt}
 \begin{aligned} Z_t:=\exp\bigg(-\int_0^tu(s)dB_s-\frac{1}{2}\int_0^t&|u(s)|^2ds+\int_0^t\int_{\R^\ell}\log(1-\theta(s,z))\tilde{N}(ds,dz) \\    &+\int_0^t\int_{\R^\ell}\brac{\log(1-\theta(s,z))+\theta(s,z)}\nu(dz)ds\bigg)
 \end{aligned}
\end{equation}
     with
    \begin{equation}\label{hTransform-u,theta}
        \begin{cases}
            u(t):=-\one_{\set{\tau_n^0\geq t\textup{ ev.}}}\sigma^T\nabla\log h(t,X_{t})\textup{ for }t\in[0,T], \\
            \displaystyle \theta(t,z) := 1-\one_{\set{\tau_n^0\geq t\textup{ ev.}}}\frac{h(t,X_{t-}+\gamma(t,X_{t-},z))}{h(t,X_{t-})}\textup{ for }t\in[0,T],\,z\in\R^\ell,
        \end{cases}
    \end{equation}  
    where  $\set{\tau_n^0\geq t\textup{ ev.}}:=\bigcup_{n\geq 1}\set{\tau_n^0\geq t}$ for $\set{\tau_n^0}_{n\geq 1}$ being the sequence of stopping times defined in \eqref{def_tau_n}. Then, $\set{Z_t}_{[0,T]}$ is a $\set{\cF_t}$-local-martingale under $\frac{\one_{\set{X_0\in A_0}}}{r_0}\bP$. 
   
   Further, if $\bP^h$ is the $h$-transform of $\bP$ as defined in \eqref{Definition:h-transform,general case with tau0}, then \begin{equation}\label{def_Q^h}  \bP^h=Z_T\frac{\one_{\set{\tau_0>T}}}{r_0}\bP.
    \end{equation}
   Hence if $\bP(\tau^0>T)=r_0$, then $\bP^h$ is the Girsanov transform of $\frac{\one_{\set{\tau^0>T}}}{r_0}\bP$ with the above coefficients $u(t),\theta(t,z)$.
\end{thm}
\begin{proof}
We first remark that for every $n\geq 1$ and $t\in[0,T]$, $h(s,X_{s-})$ is bounded away from $0$ uniformly in $s\in[0,t]$ along every sample path in $\set{\tau_n^0\geq t}$, and hence $u(t)$ and $\theta(t,z)$ in \eqref{hTransform-u,theta} are well defined. Since $h$ is of class $C^{1,2}$ on $[0,T]\times\R^n$ and $\set{X_t}_{[0,T]}$ has c\`adl\`ag sample paths, it follows that
\begin{equation*}
    \sup_{t\in[0,T]}|u(t)|<\infty,\,\sup_{t\in[0,T],|z|\geq1}|\theta(t,z)|<\infty,\text{ and }\sup_{t\in[0,T],|z|\leq 1}\abs{\frac{\theta(t,z)}{\gamma(t,X_{t-},z)}}<\infty \text{ a.s.}
\end{equation*}
Furthermore, for every $n\geq1$,  $\set{\one_{\set{\tau_n^0\geq t}}}_{[0,T]}$ is a predictable process with respect to the underlying filtration \cite[Chapter 
 V, Section 8]{rogers2000diffusionsVOL2}, which implies that $\set{u(t)}_{[0,T]}$ and $\set{\theta(t,z)}_{[0,T]}$ are also predictable. 

Given a function $h$ as in the statement of the theorem, let $\cA_h$ be as in \eqref{definition_A_h}, and define $f=-\log h$ on $\cA_h$. Then $f$ is of class $C^{1,2}$ on $\cA_h$. By direct computations, we have
\begin{equation*}
     \cL f =-\frac{1}{h}\cL h+ \frac{1}{2}|\sigma^T\nabla\log h|^2 + \int_{\R^\ell} \brac{\frac{h(\cdot+\gamma)-h}{h}-\log\frac{h(\cdot+\gamma)}{h}}\nu(dz).
\end{equation*}
Combining \eqref{hTransform-u,theta} with It\^o's formula (e.g. \cite[Theorem 4.4.7]{applebaum_2009}) applied to $f$, while incorporating the stopping times $\set{\tau_n^0}_{n\geq1}$, we have that for every $n\geq1$,
\begin{equation}\label{NegativeLogExpanded}
\begin{aligned}
    \log\pran{\frac{h(t\wedge\tau_n^0,X_{t\wedge\tau_n^0})}{h(0,X_0)}}=&\int_0^{t\wedge\tau_n^0}u(s)\,dB_s-\frac{1}{2}\int_0^{t\wedge\tau_n^0}|u(s)|^2\,ds\, \\
    &+\int_0^{t\wedge\tau_n^0}\int_{\R^\ell} \log(1-\theta(s,z))\tilde{N}(ds,dz)\\
    &+\int_0^{t\wedge\tau_n^0}\int_{\R^\ell} \brac{\theta(s,z)+\log(1-\theta(s,z))}\nu(dz)ds.
\end{aligned}
\end{equation}
By the remark made at the beginning of the proof, it is guaranteed that all the stochastic integrals in \eqref{NegativeLogExpanded} are well defined and  \eqref{NegativeLogExpanded} holds a.s. On the other hand, if we define $Z_t$ according to \eqref{Equation:Girsanov Zt}, then it follows from \eqref{NegativeLogExpanded} that, for every $n\geq 1$, \[Z_{t\wedge\tau^0_n}=\frac{h(t\wedge\tau^0_n,X_{t\wedge\tau^0_n})}{h(0,X_0)}\text{ for }t\in[0,T].\] Since $h$ satisfies the mean-value property, $\set{Z_t}_{[0,T]}$ is a local martingale under $\frac{\one_{\set{X_0\in A_0}}}{r_0}\bP$. Note that this proof of $\set{Z_t}_{[0,T]}$ being a local martingale bypasses the usual integrability requirements on $u(t)$ and $\theta(t,z)$.

Since $\set{\tau^0>T}\subseteq\set{\tau^0_n\leq t\text{ ev.}}$ for every $t\in[0,T]$, the relation \eqref{def_Q^h} follows immediately from \eqref{Definition:h-transform,general case with tau0} and \eqref{NegativeLogExpanded}.
Now assume that $r_0=\bP(\tau^0>T)$. Then, $\frac{\one_{\set{\tau^0>T}}}{r_0}\bP$ becomes a path measure, under which $\set{Z_t}_{[0,T]}$ becomes an $\set{\cF_t}$-martingale. Furthermore, by Proposition \ref{prop:equiv represent of bP^h},\[
\frac{1}{r_0}\bE_\bP\brac{Z_T\one_{\set{\tau^0>T}}}=\frac{1}{r_0}\bE_\bP\brac{\frac{h(T,X_T)}{h(0,X_0)}\one_{\set{\tau^0>T}}}=1.
\] We conclude that $\bP^h$ is the Girsanov transform of $\frac{\one_{\set{\tau^0>T}}}{r_0}\bP$ (e.g., \cite{oksendal:hal-02411121}, Theorem 1.33) with the given $u(t)$ and $\theta(t,z)$.
\end{proof}
Having identified  $\bP^h$ as a particular Girsanov transformation, we are ready to state the SDE that $\set{X_t}_{[0,T]}$ satisfies under $\bP^h$. 
\begin{cor}\label{Theorem:SDEunderPh}
Under the same setting as in Theorem \ref{GirsanovHTransform}, $(\set{X_t}_{[0,T]},\bP^h)$ is a solution to the following SDE:
\begin{equation}\label{hTransformSDE}
\begin{aligned}
   dX_t=b^h(t,X_t)dt&+\sigma(t,X_t)dB^h_t \\
   &+ \int_{|z|\leq 1}\gamma(t,X_{t-},z)\tilde{N}^h(dt,dz) 
   + \int_{|z|>1}\gamma(t,X_{t-},z)N(dt,dz),
   \end{aligned}
\end{equation}
where $b^h$ is as defined in \eqref{definition_b^h}, $\set{B^h_t}$ is a standard $\bP^h$-Brownian motion in $\R^m$, and  $\tilde{N}^h(dt,dz)$ is the $\bP^h$-compensated Poisson measure with intensity with $\nu^h$ as in \eqref{definition_nu^h}.
\end{cor}
\begin{proof}
We have established in the previous theorem that $\bP^h=Z_T\frac{\one_{\set{\tau^0>T}}}{r_0}\bP$. On $\set{\tau^0>T}$, $h(t,X_t)>0$ and $h(t,X_{t-})>0$ for all $t\in[0,T]$.
Thus, in performing the Girsanov transform, we can replace the coefficients $u(t),\theta(t,z)$ 
in \eqref{hTransform-u,theta} by
\begin{equation*}
        \begin{cases}
            u(t,x):=-\sigma^T\nabla\log h(t,x)\text{ for }(t,x)\in\cA_h, \\
            \displaystyle \theta(t,x,z) := 1-\frac{h(t,x+\gamma(t,x,z))}{h(t,x)}\text{ for }(t,x)\in\cA_h\text{ and }z\in\R^\ell,
        \end{cases}   
\end{equation*}
and define $\set{Z_t}_{[0,T]}$ accordingly, where all the stochastic integrals involved are well defined as remarked in the proof above. 

By the Girsanov theorem, we have that $dB^h_t:=u(t,X_t)\, dt + dB_t$ is a $\bP^h$-Brownian motion and $\tilde{N}^h(dt,dz):=N(dt,dz)-(1-\theta(t,X_t,z))\nu(dz)dt$ is the $\bP^h$-compensated Poisson random measure of $N(dt,dz)$. Then, the SDE \eqref{hTransformSDE} simply follows from re-writing and re-grouping the terms in the SDE \eqref{ReferenceLevyIto} to incorporate $B^h_t$ and $\tilde{N}^h(dt,dz)$.
\end{proof}
\section{SBP for Jump Diffusions: the General Theory}\label{Section:RelationSBP-htransform}
In this section, we return to the Schr\"odinger bridge problem (SBP) for jump diffusions and study its solution under the theory of $h$-transforms developed in Section \ref{Section:hTransform}. Recall from Theorem \ref{Theorem:fg-existence} that, given a Markov reference measure $\bR$ and target endpoint distributions $\rho_0,\rho_T$, if $(\f,\g)$ is the solution to the Schr\"odinger system \eqref{Equation:fgSchrodingerSystem}, then the solution to the SBP \eqref{Definition:DynamicSBP} is given by $\hat{\bP}=\f(X_0)\g(X_T)\bR$. We aim at connecting $\hat{\bP}$ with a certain $h$-transform of the underlying path measure.

In the following subsections, we first discuss a particular type of SBPs that are explicitly ``solvable''. Namely, when $\rho_0$ is a point mass, we can in fact explicitly determine $(\f,\g)$ and hence study the property of $\hat\bP$ directly. However, for general $\rho_0,\rho_T$, it is difficult, if not impossible, to derive concrete expressions for $(\f,\g)$. Instead, we will approach $\hat{\bP}$ from an $h$-transform point of view by considering the function $\h(t,x):=\int \g(y)P_{t,T}(x,dy)$ and apply the results developed in Section \ref{Section:hTransform}. Although in general we do not have any regularity condition on $\g$ \textit{a priori} to guarantee that $\h$ is harmonic, we can establish an approximation theory under Assumption \ref{Assumption:Hypo}. In particular, we will prove that the desired SBP solution $\hat{\bP}$ is achieved as the limit, under strong convergence, of a family of path measures all of which are $h$-transforms by harmonic functions. 
\subsection{SBP with Explicit Solution}\label{Section:Poisson-Example}
Let us consider the particular case of SBP when the target initial endpoint distribution is a point mass, meaning $\rho_0 =\delta_{x_0}$ for some $x_0\in\R^n$. In this scenario, the study of the SBP is restricted to a reference measure $\bR$ under which the process has the initial distribution $X_0=x_0$ a.s., because otherwise $\KL{\bQ}{\bR}=\infty$ for any admissible $\bQ$, leading to the SBP having no solution. Under the assumption that $\rho_0=\bR_0=\delta_{x_0}$, the Schr\"odinger system takes the form 
\begin{equation*}
    \begin{cases}
\vspace{0.2cm}\displaystyle \f(x_0)\Esub{\bR}{\g(X_T)} = 1, \\
\displaystyle \g(y)\f(x_0) = \frac{d\rho_T}{d\bR_T}(y)\text{ for $\bR_T$-a.e. $y$}.
\end{cases}
\end{equation*}
Assume without loss of generality that $\f(x_0)=1$ and hence $\bE_{\bR}\brac{\g(X_T)}=1$ and $\g=\frac{d\rho_T}{d\bR_T}$. As a result, the Schr\"odinger bridge is given by \[\hat\bP=\g(X_T)\bR=\frac{d\rho_T}{d\bR_T}(X_T)\bR.\] 

To make it concrete, we will examine a simple example that falls outside the usual literature on diffusion SBPs. We also illustrate the results from Section \ref{Section:hTransform} with this example. \\
\noindent \textit{Example.} Let $\bR$ be the path measure such that $\set{X_t}_{[0,T]}$ is a simple Poisson jump process with constant rate $\lambda>0$ and initial distribution $X_0\equiv0$. This process has a terminal distribution $X_T\sim\mathrm{Poisson}(\lambda T)$. We will solve the (dynamic) SBP \eqref{Definition:DynamicSBP} for $\bR$ with the target endpoint distributions $\rho_0=\delta_0$ and $\rho_T=\mathrm{Poisson}(\mu T)$, where $\mu>0$ is a constant and $\mu\neq\lambda$. Obviously, $\frac{d\rho_T}{d\bR_T}(y)=e^{(\lambda-\mu)T}(\mu/\lambda)^y$ for $y\in\N$, so according to the above, the SBP solution is $\hat\bP=e^{(\lambda-\mu)T}(\mu/\lambda)^{X_T}\bR$. Now we turn to the dynamics of $\hat\bP$.

For the given $\bR$, the SDE \eqref{ReferenceLevyIto} can be written as $dX_t=\lambda dt + \int_{|z|\leq 1} z\tilde{N}(dt,dz)$ with $\tilde N(dt,dz)$ being the compensated simple Poisson random measure with rate $\lambda$\footnote{In fact, the representation above can be reduced to $dX_t=\int_{\R}zN(dt,dz)$ since $\nu=\lambda \delta_1$.}, and the operator \eqref{ReferenceGenerator} simply becomes \begin{equation*}
    \cL f(t,x):=\frac{\partial f}{\partial t}(t,x) +\lambda \pran{f(t,x+1)-f(t,x)}\text{ for }f\in C_c^{1,2}.
\end{equation*}
To link $\hat\bP$ to an $h$-transform, we need the ``candidate'' function $\h$, in this case given by
\begin{equation*}
    \h(t,x)=e^{(\lambda-\mu)T}\pran{\frac{\mu}{\lambda}}^x\sum_{n=0}^\infty e^{-\lambda(T-t)}\frac{[\lambda(T-t)(\mu/\lambda)]^n}{n!}=e^{(\lambda-\mu)t}\pran{\frac{\mu}{\lambda}}^x.
\end{equation*}
As a function on $[0,T]\times\R^+$, $\h$ satisfies the mean-value property, is strictly positive and of class $C^{1,2}$. Furthermore, for every $(t,x)\in[0,T]\times\R^+$, \begin{equation*}
    \cL \h(t,x)=(\lambda-\mu)\h(t,x)+\lambda\pran{\frac{\mu}{\lambda}-1}\h(t,x)=0,
\end{equation*}
which means that $\h$ is harmonic. We observe that $\hat\bP=\frac{\h(T,X_T)}{\h(0,X_0)}\bR$. In other words, $\hat\bP$ is exactly the $h$-transform of $\bR$ by the harmonic function $\h$. We are ready to put the results from Section \ref{Section:hTransform} into practice.

By Theorem \ref{Theorem:MartingaleProblem-Lh}, the operator associated with $\hat\bP$ is given by: for $f\in C_c^{1,2}$,
\[
\cL^\h
f(t,x)=\frac{\partial f}{\partial t}(t,x)+(b^\h f')(t,x)+\int_{|z|\leq 1}\pran{f(t,x+z)-f(t,x)-zf'(x)}\nu^{\h}(dz);\] with $b^{\h}(t,x)=\lambda\frac{\h(t,x+1)}{\h(t,x)}=\mu$ and $\nu^{\h}=\lambda\frac{\h(t,x+1)}{\h(t,x)}\delta_1=\mu\delta_1$. In this case, the above reduces to \[
\cL^\h
f(t,x)=\frac{\partial f}{\partial t}(t,x)+\mu\pran{f(t,x+1)-f(t,x)}.\]
This indicates that $\hat\bP$ is exactly the law of the simple Poisson jump process with rate $\mu$.

We may also put this example under the lens of the Girsanov transformation, where the concerned coefficients become $u(t,x)=0$ and $\theta(t,x,.)=1-\frac{\h(t,x+1)}{\h(t,x)}=1-\frac{\mu}{\lambda}$. By Corollary \ref{Theorem:SDEunderPh}, the canonical process under $\hat\bP$ solves the SDE $dX_t=\mu dt+\int_{|z|\leq 1}z\tilde{N}^{\h}(dt,dz)$ where  \[\tilde{N}^{\h}(dt,dz):=N(dt,dz)-(1-\theta(t,x,z))\nu(dz)dt=N(dt,dz)-\mu\delta_1dt\]  is the compensated Poisson random measure. 
This again confirms that $\hat\bP$ is provided by a simple Poisson jump process with rate $\mu$. 
\subsection{Relation between SBP and $h$-Transform}\label{Subsec:Approximation}
In this subsection, we will give a comprehensive analysis of the connections between the solution to the SBP and the $h$-transform of the reference path measure. Under the setup adopted at the beginning of Section \ref{Section:RelationSBP-htransform}, given the strong Markov reference measure $\bR$, the target endpoint distributions $\rho_0,\rho_T$, and the solution $(\f,\g)$ to the Schr\"odinger system \eqref{Equation:fgSchrodingerSystem}, we define 
\begin{equation}\label{def:frak_h}
\h(t,x):=\int_{\R^n}\g(y)P_{t,T}(x,dy)\text{ for }(t,x)\in[0,T]\times\R^n,
\end{equation} 
where $\set{P_{s,t}(x,dy):0\leq s\leq t\leq T,x\in\R^n}$ is the transition distribution under $\bR$. Note that $\h$ is non-negative (and may take $\infty$ value) and satisfies the mean-value property \eqref{Equation:MeanValueProperty}. Moreover, since $\h(T,x)=\g(x)$ and $\h(0,x)=\bE_{\bR_{0T}}\brac{\g(y)|x}$, it follows from \eqref{Equation:fgSchrodingerSystem} that $\f(x)\h(0,x)=\frac{d\rho_0}{d\bR_0}(x)$ for $\bR_0$-a.e. $x$. Let us state a preparatory lemma.
\begin{lem}\label{lem:from bR to bP}
    Let $\bR, \rho_0,\rho_T, \f,\g,\h,$ be the same as above. Set $\displaystyle S_0:=\set{x\in \R^n:\frac{d\rho_0}{d\bR_0}(x)=0}$ and define
    \begin{equation}\label{eq:def of bP}
        \bP:=\one_{S_0^c}(X_0)\frac{d\rho_0}{d\bR_0}(X_0)\bR. 
    \end{equation}
 Then, $\bP$ is a path measure with initial endpoint (i.e., $t=0$) distribution $\rho_0$, and the  solution to the dynamic SBP \eqref{Definition:DynamicSBP} has the representation: 
\begin{equation}\label{eq:hatbP in terms of bP}
    \hat\bP= \frac{\h(T,X_T)}{\h(0,X_0)}\bP.
\end{equation}
\end{lem}
\begin{proof}
The statement on $\bP$ follows immediately from the definitions of $\bP$ and $S_0$.
Further, we observe that \[\frac{\h(T,X_T)}{\h(0,X_0)}\one_{S_0^c}(X_0)\frac{d\rho_0}{d\bR_0}(X_0)=\one_{S_0^c}(X_0)\f(X_0)\g(X_T)\text{ $\bR$-a.s.},\] 
Therefore, to establish \eqref{eq:hatbP in terms of bP}, we only need to show that
\[\hat\bP=\one_{S_0^c}(X_0)\f(X_0)\g(X_T)\bR. \]
By Theorem \ref{thm:relation static and dynamic SBP}, it is sufficient to prove that $\hat\pi$, the solution to the static SBP \eqref{Definition:StaticSBP}, can be represented as 
\begin{equation*}
    \hat\pi= \one_{S_0^c}(x)\f(x)\g(y)\bR_{0T}.
\end{equation*} 
By \eqref{Equation:pihat}, we would be done if we could confirm that $\one_{S_0}(x)\f(x)\g(y)\bR_{0T}\equiv 0$. Indeed,
\begin{equation*}
    \int_{S_0}\int_{\R^n} \f(x)\g(y)\bR_{0T}(dxdy)=\int_{S_0}\f(x)\h(0,x)\bR_0(dx)=\int_{S_0}\frac{d\rho_0}{d\bR_0}(x)\bR_0(dx)=0.
\end{equation*}
\end{proof} 
Lemma \ref{lem:from bR to bP} implies that, in solving the SBP \eqref{Definition:DynamicSBP}, by replacing $\bR$ by $\bP$ following \eqref{eq:def of bP}, it is sufficient to treat the case when the reference initial endpoint distribution matches the target one. Further, the representation \eqref{eq:hatbP in terms of bP} of $\hat{\bP}$ clearly suggests that $\hat{\bP}$ should be viewed as an $h$-transform of $\bP$. Therefore, we can transport to $\hat{\bP}$ the theory established in Section \ref{Section:hTransform}. We will restate the relevant results in the theorem below.  
\begin{thm}\label{Theorem:SBridge=frakh-transform}
    Consider the SBP \eqref{Definition:DynamicSBP} with a strong Markov reference measure $\bR$ and target endpoint distributions $\rho_0,\rho_T$. Let $(\f,\g)$ be the solution to the Schr\"odinger system \eqref{Equation:fgSchrodingerSystem}, $\h$ be as in \eqref{def:frak_h} and $\bP$ be as in \eqref{eq:def of bP}. Then, the SBP solution $\hat{\bP}$ is the $h$-transform of $\bP$ by $\h$, i.e., $\hat\bP=\bP^{\h}$ where $\bP^{\h}$ is  given by \eqref{Definition:h-transform,general case with r0}\footnote{In this case, by the definition \eqref{eq:def of bP} of $\bP$, $\h(0,X_0)>0$ a.s. under $\bP$, and hence $r_0=1$ and $\one_{\set{X_0\in A_0}}$ can be omitted in \eqref{Definition:h-transform,general case with r0}.} (or equivalently \eqref{Definition:h-transform,general case with tau0}) with $h=\h$.
    
    Now assume that $\h$ is harmonic. We have the following: 
\begin{enumerate}
    \item If the operator $L$ in \eqref{ReferenceGenerator} satisfies \textup{\ref{Assumption:MartingaleProblem-s,x}} and $\bP$ is the solution to the martingale problem for $(L,\rho_0)$, then $\hat\bP$ is the solution to the martingale problem for $(L^{\h},\rho_0)$, where $L^{\h}$ is given by \eqref{Definition:L^h} with $h=\h$. 
    \item If the SDE \eqref{ReferenceLevyIto} satisfies \textup{ \ref{Assumption:SDEsolution}} and $(\set{X_t}_{[0,T]},\bP)$ is a solution to \eqref{ReferenceLevyIto} with $X_0\sim\rho_0$, then $(\set{X_t}_{[0,T]},\hat\bP)$ is a solution to the SDE: \begin{equation*}
    \begin{aligned}
       \hspace{0.5cm} dX_t=b^{\h}(t,X_{t-})&dt+\sigma(t,X_{t-})dB_t \\
       &+ \int_{|z|\leq 1}\gamma(t,X_{t-},z)\tilde{N}^{\h}(dt,dz)+\int_{|z|>1}\gamma(t,X_{t-},z)N(dt,dz)
    \end{aligned}
    \end{equation*}
    with $X_0\sim\rho_0$, where $b^{\h}$ is as in \eqref{definition_b^h} with $h=\h$, $\tilde{N}^{\h}(dt,dz)=N(dt,dz)-\nu^{\h}(dz)dt$, and $\nu^{\h}$ is given by \eqref{definition_nu^h} with $h=\h$.
\end{enumerate}
\end{thm}
This theorem covers the ``good case'', where the function $\h$ in \eqref{def:frak_h} is harmonic. However, the harmonic property of $\h$ is not guaranteed in general. We recall that $\h$ is derived from $\g$ and, as one of the pair of solutions to the Schr\"odinger system, $\g$ is only assumed to be non-negative and measurable. For diffusion models, the harmonic property of $\h$ is known to hold under mild regularity conditions on the coefficients $b,\sigma$ \cite{jamison1975markov,dai1991stochastic}. On the other hand, analogous results in the jump-diffusion literature are rather scarce. Consequently, to treat the case where $\h$ is not known to be harmonic, we will invoke Assumption \ref{Assumption:Hypo} and construct a sequence of harmonic functions $\set{h_k}_{k\geq1}$ converging to $\h$ such that the sequence of $h$-transforms $\bP^{h_k}$ converges to $\hat\bP$ in the strong topology. \\

Let us begin by building an approximation to $\g$. First, since $\g$ is non-negative and measurable, there exists a sequence of bounded simple functions $\set{g'_k}_{k\geq 1}$ such that $g'_k$ increasingly converges to $\g$ as $k\to \infty$ and
\begin{equation*}
    g'_k:=\sum_{i=1}^{N^{(k)}}c_i^{(k)}\one_{A_i^{(k)}},
\end{equation*}
where $0 \leq c_i^{(k)} \leq k$ and the sets $A_i^{(k)}$ are measurable and pairwise disjoint for different $i$'s. Fix $k$. By Lusin's theorem, there exists $g''_k\in C_b(\R^n)$ such that         $\bR_T(\set{g'_k \neq g''_k}) \leq  k^{-3}$. We further choose a compact set $D_k\subseteq \R^n$ such that $\bR_T(D_k^c) \leq k^{-3}$ and choose $g_k\in C_c^\infty(\R^n)$ such that $\norm{g_k}_{\text{u}}\leq k$ and $\norm{g_k-g''_k}_{\text{u},D_k} \leq k^{-2}$. Let $f_k$ be the non-negative function on $\R^n$ such that, for every $x\in\R^n$,  $f_k(x):=\pran{\bE_{\bR_{0T}}\brac{g_k(y)|x}}^{-1}$ if the denominator is positive, and $f_k(x)=0$ otherwise.
\begin{lem}\label{lem:f_k,g_k approximation convergence}
    Let $f_k,g_k : k \geq 1$ be defined as above. Then
    \begin{equation*}
        \lim_{k\to\infty}f_k(X_0)g_k(X_T) \frac{d\rho_0}{d\bR_0}(X_0)= \one_{S_0^c}(X_0)\f(X_0)\g(X_T)\,\, \text{ $\bR$-a.s.},
    \end{equation*}
    or equivalently,
    \begin{equation*}
        \lim_{k\to\infty}f_k(x)g_k(y) \frac{d\rho_0}{d\bR_0}(x)= \one_{S_0^c}(x)\f(x)\g(y)\,\, \text{ $\bR_{0T}$-a.s. }
    \end{equation*}
\end{lem} 
\begin{proof}
    First note that, by the Borel-Cantelli lemma, $\bR_T(\set{g'_k \neq g''_k} \text{ i.o.})=0$ and therefore
    \begin{equation}\label{eq:g''_k and g'_k limit}
        \lim_{k\to\infty}g''_k=\lim_{k\to\infty}g'_k=\g\,\, \bR_{T}\text{-a.s.}
    \end{equation}  Further,
    \begin{equation*}
    \bE_{\bR_{T}}\brac{\abs{g_k-g''_k}} \leq 2k\bR_T(D_k^c)+\frac{1}{k^2}\bR_T(D_k)\leq \frac{3}{k^2},
    \end{equation*}
    and thus for any $\eps>0$, \[\bR_{T}\pran{\set{\abs{g_k-g''_k}>\eps}}\leq \frac{3}{\eps k^2}.\] Applying the Borel-Cantelli lemma again, we have that
    \begin{equation}\label{eq:g_k and g''_k limit}
        \lim_{k\to\infty}\abs{g_k-g''_k}=0\,\,\bR_{T}\text{-a.s.}
    \end{equation}
   Combining \eqref{eq:g''_k and g'_k limit} and \eqref{eq:g_k and g''_k limit}, we conclude that $\lim_{k\to\infty}g_k=\g\,\,\bR_{T}\text{-a.s.}$, which is equivalent to \[\lim_{k\to\infty}g_k(y)=\g(y)\,\,\bR_{0T}\text{-a.s.}\] 
    It remains to show that 
    \begin{equation*}
        \lim_{k\to\infty}f_k(x)\frac{d\rho_0}{d\bR_0}(x)=\one_{S_0^c}(x)\f(x)\,\,\bR_{0T}\text{-a.s.}
    \end{equation*} Recall that since $\f$ solves the Schr\"odinger system,
    \begin{equation}
     \label{eq:frak_f solve Schrodinger sys}   \f(x)\bE_{\bR_{0T}}\brac{\g(y)|x}=\frac{d\rho_0}{d\bR_0}(x) \,\,\bR_{0T}\text{-a.s.},
    \end{equation} 
    and therefore from the definition of $S_0$ and $f_k$, it suffices to show that  
    \begin{equation}\label{eq:lim of conditional g_k}
        \lim_{k\to\infty}\bE_{\bR_{0T}}\brac{g_k(y)|x} = \bE_{\bR_{0T}}\brac{\g(y)|x} \,\,\bR_{0T}\text{-a.s.}
    \end{equation}
    By the monotone convergence of conditional expectation, \[\lim_{k\to\infty}\bE_{\bR_{0T}}\brac{g'_k(y)|x} = \bE_{\bR_{0T}}\brac{\g(y)|x}.\] Note that for every $k\geq 1$, 
    \begin{align*}
        \bE_{\bR_{0T}}\brac{\abs{g_k(y)-g'_k(y)}}
            &\leq \bE_{\bR_{T}}\brac{\abs{g_k-g''_k}}+\bE_{\bR_{T}}\brac{\abs{g'_k-g''_k}}\\
            &\leq \frac{3}{k^2} + k\bR_T\pran{\set{g'_k\neq g''_k}}\\
            & = \frac{4}{k^2},
    \end{align*}
    which implies that  \[\lim_{k\to\infty}\abs{\bE_{\bR_{0T}}\brac{g_k(y)|x} - \bE_{\bR_{0T}}\brac{g'_k(y)|x}}=0 \;\;\bR_{0T}\text{-a.s.}\]
    From here \eqref{eq:lim of conditional g_k} follows immediately.
\end{proof}
We can now express the SBP solution $\hat\bP$ as the limit of a sequence of $h$-transforms by harmonic functions. 
\begin{thm}\label{thm:P^hk strong conv to hatP}
    Suppose that Assumption \textup{\ref{Assumption:Hypo}} holds for $\set{P_{s,t}(x,\cdot):0\leq s\leq t\leq T,x\in\R^n}$ the transition distribution  under the strong Markov reference measure $\bR$. Let $\bP,\hat\bP$ be the same as in Lemma \ref{lem:from bR to bP} and $g_k,f_k$, $k\geq1$ be as in Lemma \ref{lem:f_k,g_k approximation convergence}. For every $k\geq 1$, we define 
\begin{equation}\label{eq:def of h_k}
    h_k(t,x):=\int_{\R^n}g_k(y)P_{t,T}(x,dy)\text{ for }(t,x)\in[0,T]\times\R^n,
\end{equation}
and 
\begin{equation}\label{eq:def of bP^hk}
    \bP^{h_k}:=\frac{\one_{\set{X_0\in A_k}}}{r_k}\frac{h_k(T,X_T)}{h_k(0,X_0)}\bP,
\end{equation}
where $A_k:=\set{x\in\R^n:h_k(0,x)>0}$ and $r_k:=\rho_0(A_k)$. 
Then $\set{\bP^{h_k}}_{k \geq 1}$ is a sequence of $h$-transforms by harmonic functions, and  $\bP^{h_k}$ converges to $\hat\bP$ in the strong topology in the sense that \[\lim_{k\to\infty}\sup_{A\in\cF}|\bP^{h_k}(A)-\hat{\bP}(A)|=0.\]
\end{thm}
\begin{proof}
Fix $k \geq 1$. Since $g_k \in C_c^\infty$ and \ref{Assumption:Hypo} holds, $h_k$ is harmonic. Note that $h_k(T,y)=g_k(y)$ for $y\in\R^n$ and $h_k(0,x)=\bE_{\bR_{0T}}\brac{g_k(y)|x}=(f_k(x))^{-1}$ for $x\in A_k$. By combining (\ref{eq:def of bP^hk}) with (\ref{eq:def of bP}), we can rewrite $\bP^{h_k}$ as 
\[
\bP^{h_k}=\frac{\one_{\set{X_0\in A_k}}}{r_k}\frac{h_k(T,X_T)}{h_k(0,X_0)}\frac{d\rho_0}{d\bR_0}(X_0)\bR=\frac{1}{r_k}f_k(X_0)g_k(X_T)\frac{d\rho_0}{d\bR_0}(X_0)\bR.
\]

Recall that $\hat{\bP}=\one_{S_0^c}(X_0)\f(X_0)\g(X_T)\bR$. Therefore, to prove the conclusion of the theorem, it is sufficient to show that $\bP^{h_k}$ is a path measure for each $k\geq1$ and 
\begin{equation}\label{eq:fkgk conv in L1}
\frac{1}{r_k}f_k(X_0)g_k(X_T)\frac{d\rho_0}{d\bR_0}(X_0)\xrightarrow{k\to\infty}\one_{S_0^c}(X_0)\f(X_0)\g(X_T)\text{ in }L^1(\bR).
\end{equation} 
First, we apply Lemma \ref{lem:f_k,g_k approximation convergence} to get 
\begin{equation}     \label{eq:fkgk conv a.s.}      \lim_{k\to\infty}f_k(X_0)g_k(X_T)\frac{d\rho_0}{d\bR_0}(X_0)=\one_{S_0^c}(X_0)\f(X_0)\g(X_T)\,\,\bR\text{-a.s.}
\end{equation}
Next, we derive that 
\begin{equation*}
\begin{aligned}   \bE_{\bR}&\brac{f_k(X_0)g_k(X_T)\frac{d\rho_0}{d\bR_0}(X_0)} \\
&=\bE_{\bR}\brac{f_k(X_0)\bE_{\bR}\brac{g_k(X_T)|X_0}\frac{d\rho_0}{d\bR_0}(X_0)}\\
&=\bE_{\bR}\brac{\one_{\set{X_0\in A_k}}\frac{d\rho_0}{d\bR_0}(X_0)}\text{ since $f_k=(\bE_{\bR_{0T}}\brac{g_k(y)|\cdot})^{-1}$ on $A_k$} \\
&=\rho_0(A_k) = r_k,
\end{aligned}
\end{equation*} which confirms that $\bP^{h_k}$ is a path measure.
Furthermore, by (\ref{eq:frak_f solve Schrodinger sys}) and (\ref{eq:lim of conditional g_k}), for $\bR_0$-a.e. (and hence $\rho_0$-a.e.) $x\in \R^n$, if $\frac{d\rho_0}{d\bR_0}(x)>0$, then $\bE_{\bR_{0T}}\brac{\g(y)|x}>0$ and thus \[\lim_{k\to\infty}h_k(0,x)=\lim_{k\to\infty}\bE_{\bR_{0T}}\brac{g_k(y)|x}>0.\] Therefore, using Fatou's lemma, we have that
\begin{equation*}
    1=\rho_0\pran{\set{\frac{d\rho_0}{d\bR_0}>0}}\leq \rho_0\pran{\set{\lim_{k\to\infty}h_k(0,x)>0}}
        \leq \liminf_{k\to\infty}\rho_0(A_k)=r_k.
\end{equation*}
We conclude that as $k\to\infty$, 
\begin{equation} \label{eq:fkgk conv in expectation}  
  \bE_{\bR}\brac{f_k(X_0)g_k(X_T)\frac{d\rho_0}{d\bR_0}(X_0)}=r_k \xrightarrow{k\to\infty}  1 = \bE_{\bR}\brac{\f(X_0)\g(X_T)}.
\end{equation}
Finally, combining (\ref{eq:fkgk conv a.s.}) and (\ref{eq:fkgk conv in expectation}), the desired $L^1(\bR)$ convergence relation (\ref{eq:fkgk conv in L1}) follows from Scheff\'e's lemma. 
\end{proof}
We close this section with a result stating that the above strong convergence does not only occur at the endpoints $t=0,T$, but also in a ``dynamic'' manner in terms of the approximation of $\h$.
\begin{thm}
Under the assumptions of  Theorem \ref{thm:P^hk strong conv to hatP}, if $\set{h_k}_{k \geq 1}$ is the sequence of harmonic functions defined in \eqref{eq:def of h_k}, then for every $t\in[0,T]$,  
    \begin{equation*}
        \lim_{k\to\infty}h_k(t,X_t) = \h(t,X_t)\,\,\bP\text{-a.s.,}
    \end{equation*}
and 
\begin{equation*}
   \frac{\one_{\set{X_0\in A_k}}}{r_k}\frac{h_k(t,X_t)}{h_k(0,X_0)}\xrightarrow{k\to\infty}\frac{\h(t,X_t)}{\h(0,X_0)}\text{ in $L^1(\bP)$.}
\end{equation*}
\end{thm} 
\begin{proof}
Let $g_k,g'_k$, $k\geq1$, be the same as introduced before Lemma \ref{lem:f_k,g_k approximation convergence}. For every $k\geq 1$, and every $(t,x)\in[0,T]\times\R^n$, define \[h'_k(t,x):=\int_{\R^n}g'_k(y)P_{t,T}(x,dy).\] Clearly, $\set{h'_k}_{k\geq1}$ is an increasing sequence that converges to $\h$. Moreover,
\begin{align*}
    \bE_{\bR}\brac{\abs{h_k(t,X_t)-h'_k(t,X_t)}} &= \int_{\R^n}\int_{\R^n}\abs{h_k(t,x)-h'_k(t,x)}P_{0,t}(w,dx)\bR_0(dw)\\
    &\leq \int_{\R^n}\int_{\R^n}\int_{\R^n}\abs{g_k(y)-g'_k(y)}P_{t,T}(x,dy)P_{0,t}(w,dx)\bR_0(dw)\\
    &=\bE_{\bR_T}\brac{\abs{g_k-g'_k}}\\
    &\leq \frac{4}{k^2}\text{ by the same argument as in the proof of Lemma \ref{lem:f_k,g_k approximation convergence}}.
\end{align*}
It is sufficient for us to conclude that  $\lim_{k\to\infty}\abs{h_k(t,X_t)-h'_k(t,X_t)}=0$ $\bR$-a.s., and thus
\begin{equation*}
    \lim_{k\to\infty}h_k(t,X_t) = \h(t,X_t)\,\,\bR\text{-a.s., and hence $\bP$-a.s.}
\end{equation*}
To prove the second statement, we observe that
\begin{equation*}
    \bE_{\bP}\brac{\frac{\one_{\set{X_0\in A_k}}}{r_k}\frac{h_k(T,X_T)}{h_k(0,X_0)}\bigg|\cF_t}=\frac{\one_{\set{X_0\in A_k}}}{r_k}\frac{h_k(t,X_t)}{h_k(0,X_0)}
\end{equation*}
and
\begin{equation*}
    \bE_{\bP}\brac{\frac{\h(T,X_T)}{\h(0,X_0)}\bigg|\cF_t}=\frac{\h(t,X_t)}{\h(0,X_0)}.
\end{equation*}
It is easy to see that the fact \eqref{eq:fkgk conv in L1} established in Theorem \ref{thm:P^hk strong conv to hatP} is equivalent to
\[\frac{\one_{\set{X_0\in A_k}}}{r_k}\frac{h_k(T,X_T)}{h_k(0,X_0)}\xrightarrow{k\to\infty}\frac{\h(T,X_T)}{\h(0,X_0)}\text{ in }L^1(\bP).\]
Therefore, the respective conditional expectations must also satisfy the same $L^1(\bP)$ convergence relation.
\end{proof}
\section{SBP for Jump Diffusions: when Density Exists}\label{Section:smooth-coefficients}
So far we have been studying the SBP for general jump diffusions under Assumptions \ref{Assumption:MartingaleProblem-s,x} (or \ref{Assumption:SDEsolution}) and \ref{Assumption:Hypo}. These assumptions, especially  \ref{Assumption:Hypo}, are better understood for jump diffusions that admit densities. A canonical scenario in which a density function may exist is when the jump diffusion has sufficiently regular coefficients. Such models have been well studied in the context of \textit{stochastic flows} with a fast growing literature \cite{ishikawa2006malliavin,cass2009smooth,kunita2011analysis,kunita2013nondegenerate,kunita2019stochastic}. 

In this section, we first revisit the Schr\"odinger system for a jump-diffusion SBP under the general assumption that the transition density exists. Then, we adopt the setting of Kunita \cite{kunita2019stochastic} to study the SBP for certain jump-diffusion stochastic flows, where, under the presence of smooth\footnote{To be precise, the density functions are $C^1$ in temporal variables and $C^\infty$ in spatial variables.} density functions, we obtain results analogous to the diffusion case on the Schr\"odinger system solution $(\varphi,\hat\varphi)$ as reviewed in Section \ref{Section:DiffusionReview}. By doing so, we generalize in various degrees the known results about the diffusion SBP to the jump-diffusion setting. In addition, based on an approximation method, we extend our study to the SBP for jump diffusions with stable-like jump components. For this type of models, the density is known to exist in many situations, including when the coefficients are less regular \cite{czq2016heatkernel,chen2016heat,chen2020heat,jin2017heat}. 
\subsection{Schr\"odinger System for Jump Diffusions\label{subsection:varphi_hat varphi via f_g}} 
As in the previous section, we continue working under the assumption that the conclusion of Theorem \ref{Theorem:fg-existence} holds and $(\f,\g)$ is a pair of non-negative, measurable functions that solve the Schr\"odinger system (\ref{Equation:fgSchrodingerSystem}). As in the diffusion case, we assume that the reference endpoint distributions  $\bR_0,\bR_T$ and the target ones $\rho_0,\rho_T$ all admit probability density functions, which we will write as $\bR_0(x)dx$, etc, and the transition distribution under $\bR$ admits a transition density function everywhere and at all time, i.e., for every $0\leq s<t\leq T$ and $x\in \R^n$,  $P_{s,t}(x,dy)=p(s,x,t,y)dy$. We will re-derive the Schr\"odinger system for the jump-diffusion SBP. Note that for this derivation, we do not need the transition density function to be strictly positive.

Recall that the solution to the SBP takes the form $\hat{\bP}=\f(X_0)\g(X_T)\bR$. Expressed in terms of the density functions, this fact yields that 
\begin{equation}\label{eq:fg integrable under R_0T}
\int_{\R^n}\int_{\R^n} \f(x)\g(y)p(0,x,T,y)\bR_0(x)dxdy=1.
\end{equation}
Further, the Schr\"odinger system \eqref{Equation:fgSchrodingerSystem} can be rewritten as
\begin{equation}\label{eq:Schrodinger system under density}
   \begin{cases}
    \vspace{0.2cm}\displaystyle \f(x)\bR_0(x)\pran{\int_{\R^n}\g(y)p(0,x,T,y)dy}=\rho_0(x) & \text{ for $\bR_0$-a.e. $x$},\\
    \displaystyle \g(y)\pran{\int_{\R^n}\f(x)\bR_0(x)p(0,x,T,y)dx} = \rho_T(y)& \text{ for $\bR_T$-a.e. $y$}.
    \end{cases} 
\end{equation}
Therefore, if we further define $\varphi_T:=\g$, $\hat\varphi_0:=\f\cdot\bR_0$, and 
\begin{equation*}
  \begin{cases}
            \varphi_0(x):=\int_{\R^n} p(0,x,T,y)\varphi_T(y)dy\text{ for }x\in\R^n,\\ \hat\varphi_T(y):=\int_{\R^n} p(0,x,T,y)\hat\varphi_0(x)dx\text{ for }y\in\R^n,
  \end{cases}
\end{equation*}
then \eqref{eq:Schrodinger system under density} can be interpreted as the 4-tuple of functions $(\varphi_0,\varphi_T,\hat\varphi_0,\hat\varphi_T)$ being a solution to the Schr\"odinger system as in  \eqref{phiphihat}.

It is important to note that, although (\ref{eq:fg integrable under R_0T}) guarantees that the product $\f\g$ is integrable with respect to $p(0,x,T,y)\bR_0(x)=\bR_{0T}(x,y)$, it does not imply the respective integrability of $\f,\g$ with respect to $\bR_0,\bR_T$. As a result, the functions $\varphi_0,\varphi_T,\hat\varphi_0,\hat\varphi_T$ may take the value of infinity \cite{leonard2014some}. On the other hand, the first equation in (\ref{eq:Schrodinger system under density}) implies that $0<\hat\varphi_0(x) \varphi_0(x)<\infty$ for $\rho_0$-a.e. $x\in\R^n$, which means that $0<\hat\varphi_0<\infty$ and $0<\varphi_0<\infty $ a.s. under $\rho_0$. Similarly, we also have $0<\hat\varphi_T<\infty$ and $0<\varphi_T<\infty $ a.s. under $\rho_T$.\\

We now ready present a derivation of the probability density function of the marginal distribution of $\hat{\bP}$. This is a known result that relies only on the absolute continuity of the initial distribution $\rho_0$ (with respect to the Lebesgue measure) and the existence of the transition density $p(s,x,t,y)$; we include this derivation for completeness. For every $(t,x)\in[0,T]\times\R^n$ and  $B\in \cB(\R^n)$, we have
\begin{align*}
    \hat{\bP}(X_t(\cdot)\in B)&=\bE_{\bR}\brac{\f(X_0)\g(X_T)\one_B(X_t)}\\
    &=\int_{\R^n}\f(w)\pran{\int_B \int_{\R^n} \g(y)P_{t,T}(x,dy)P_{0,t}(w,dx)}\bR_0(dw)\\ 
    &=\int_B \int_{\R^n}\int_{\R^n}\f(w)\g(y)p(t,x,T,y)\bR_0(w) p(0,w,t,x)dydwdx\\
    &=\int_B \int_{\R^n}\g(y)p(t,x,T,y)dy\int_{\R^n}\f(w)p(0,w,t,x)\bR_0(w)dwdx,
\end{align*}
showing that the marginal distribution at time $t$, denoted by $\hat\bP_t$, admits a probability density function given by 
\begin{equation}\label{Equation:SBPdensity-fg}
    \hat\bP_t(x):=\pran{\int_{\R^n}\g(y)p(t,x,T,y)dy}\pran{\int_{S_0^c}\f(w)p(0,w,t,x)\bR_0(w)dw}.
\end{equation}
If $\varphi,\hat\varphi$ are two functions on $[0,T]\times\R^n$ given by, for $(t,x)\in[0,T]\times\R^n$,
\begin{equation}\label{Equation:phi-definition}
    \varphi(t,x):=\int_{\R^n}\g(y)p(t,x,T,y)dy
\end{equation}
and
\begin{equation}\label{Equation:phihat-definition}
    \hat\varphi(t,x):=\int_{\R^n}\f(w)p(0,w,t,x)\bR_0(w)dw,
\end{equation}
then (\ref{Equation:SBPdensity-fg}) can be rewritten as $\hat\bP_t(x)=\varphi(t,x)\hat\varphi(t,x)$. Further, we observe that the boundary values $(\varphi(0,\cdot),\varphi(T,\cdot),\hat\varphi(0,\cdot),\hat\varphi(T,\cdot))$ coincide with $(\varphi_0,\varphi_T,\hat\varphi_0,\hat\varphi_T)$ introduced above.

Again, we point out that (\ref{Equation:SBPdensity-fg}) only guarantees the finiteness of the product $(\varphi\hat{\varphi})(t,x)$ for every $t\in[0,T]$ and Lebesgue-a.e. $x\in\R^n$, but $\varphi(t,x)$ and $\hat\varphi(t,x)$ alone may be infinite. However, under certain circumstances, we can obtain the finiteness of $\varphi,\hat\varphi$ individually. For example, if $\f,\g$ are bounded, then it follows immediately from (\ref{Equation:phi-definition}) and (\ref{Equation:phihat-definition}) that $\varphi$ is bounded and $\hat\varphi$ is Lebesgue-a.e. finite. Another special case is when the transition density function is \textit{positive}. Indeed, if $p(s,x,t,y)>0$ for every $0\leq s< t\leq T$ and Lebesgue-a.e. $x,y\in\R^n$, then  (\ref{Equation:phi-definition}) and (\ref{Equation:phihat-definition}) imply that, for every $0\leq t\leq T$ and Lebesgue-a.e. $x\in \R^n$, $\varphi(t,x)>0$ and $\hat\varphi(t,x)>0$, and since their product is finite, it is necessary that $0<\varphi(t,x)<\infty$ and $0<\hat\varphi(t,x)<\infty$.
\subsection{Dynamics of Schr\"odinger Bridge for Jump Diffusions}
In this subsection, we will derive the partial integro-differential equation (PIDE) formulation of the Schr\"odinger system for jump diffusions analogous to \eqref{Equation:ItoSchrodingerSystem} in the diffusion case. In order to express the dynamics of the SBP solution in terms of PIDEs, we need $\varphi$ (and $\hat\varphi$) to possess a certain level of regularity. As the regularity results are rather limited for general jump diffusions, we will restrict ourselves to certain jump diffusions whose regularity theory is better understood and we will give a careful study of the SBP associated with these jump diffusions.

To be specific, we will  follow the framework developed in \cite{kunita2019stochastic} (Chapters 4-6). Upon requiring that the $L^2(\R^n)$-adjoint operator $L^*$ exists, we will verify that the functions $\varphi,\hat\varphi$ defined in \eqref{Equation:phi-definition}, \eqref{Equation:phihat-definition} are sufficiently regular and solve the corresponding PIDEs. We will also show that in this case, the product $\varphi\hat\varphi$ satisfies $(\frac{\partial}{\partial t}+L^{\hat\bP*})(\varphi\hat\varphi)=0$, where $L^{\hat\bP*}$ is the adjoint operator of $L^{\hat\bP}$, the operator associated with the SBP solution $\hat\bP$. This is already suggested by the result from the previous subsection that $\varphi\hat\varphi$ is the density of the marginal distribution of $\hat\bP$.\\

We will start with a brief review of the theory on the adjoint operator $L^*$ developed in \cite{kunita2019stochastic}. Let $L$ be the jump-diffusion operator as in (\ref{ReferenceGenerator}).  For every $(t,z)\in\R^+\times\R^\ell$, define the \textit{jump map} $\phi_{t,z} : \R^n \to \R^n$ to be
\begin{equation*}
    \phi_{t,z}(x):=\gamma(t,x,z)+x.
\end{equation*}
We will strengthen Assumptions \ref{Assumptions B} to  the following set of conditions that impose higher regularity on the coefficients of $L$:
\paragraph{\textbf{Assumptions (C)}}\customlabel{Assumptions-(C)}{\textbf{(C)}}
\begin{enumerate}
    \item $b,\sigma\in C^{1,\infty}_b(\R^+\times\R^n)$ with $\sigma\sigma^T$ being non-negative definite everywhere.
    \item $\gamma\in C^{1,\infty,2}_b(\R^+\times\R^n\times\R^\ell)$ with $\gamma(\cdot,\cdot,0)\equiv0$.
    \item for every $(t,z)\in\R^+\times\R^\ell$, $\phi_{t,z}$ is a $C^\infty$-diffeomorphism\footnote{That is, $\phi_{t,z}:\R^n\to\R^n $ is bijective and $\phi_{t,z},\phi_{t,z}^{-1}$ are both $C^\infty$ maps.} and the function $(t,x,z)\mapsto\lambda(t,x,z):=\phi_{t,z}^{-1}(x)-x$ is in  $C^{1,\infty,2}_b(\R^+\times\R^n\times\R^\ell)$.
\end{enumerate}   
It is known that the regularity conditions above lead to the existence and the uniqueness of the solution to the SDE (\ref{ReferenceLevyIto}) (see, e.g., Theorem 3.3.1 in \cite{kunita2019stochastic}), which means that \ref{Assumption:SDEsolution} is satisfied in this case. We will show in a moment that \ref{Assumption:Hypo} also holds under these conditions. Let us first state the following adaptation of Proposition 4.6.1 in \cite{kunita2019stochastic}.
\begin{thm}\label{Theorem:DualOperator-L*}
    Assume $L$ is a jump-diffusion operator given by \eqref{ReferenceGenerator} with coefficients satisfying the Assumptions \textup{\ref{Assumptions-(C)}}. Then, the $L^2([0,T]\times\R^n)$-adjoint operator $L^*: C^{\infty}_0 \to C^{\infty}_b$ is well defined, and for every $f\in C_c^\infty([0,T]\times\R^n)$ and $(t,x)\in[0,T]\times\R^n$,
\begin{equation}\label{JumpDiffusionAdjoint}
\begin{aligned}
     L^*f(t,x)&:= -\nabla \cdot (bf) (t,x)+\frac{1}{2} \sum_{i,j}\frac{\partial^2[(\sigma\sigma^T)_{ij}f]}{\partial x_i \partial x_j}(t,x) \\
    &\;+\int_{\R^\ell} \left[\det \nabla\phi_{t,z}^{-1}(x)f(t,\phi_{t,z}^{-1}(x))-f(t,x) +\one_{|z|\leq 1}\nabla\cdot (\gamma f)(t,x,z) \right]\nu(dz),
\end{aligned}
\end{equation}
where $\det \nabla\phi_{t,z}^{-1}$ is the Jacobian determinant of $\phi_{t,z}^{-1}$. 
\end{thm}
\begin{proof} This theorem is essentially the same result as Proposition 4.6.1 in \cite{kunita2019stochastic} except that the latter requires the L\'evy measure $\nu$ to have a ``weak drift'' in the sense that $\lim_{\epsilon\to 0}\int_{\epsilon\leq |z|\leq 1}z\nu(dz)$ exists. Upon dropping the weak drift assumption, we obtain (\ref{JumpDiffusionAdjoint}) as the expression for $L^*$, which differs from the counterpart in \cite{kunita2019stochastic}. So, to prove this theorem, we only need to verify that the integral with respect to $\nu$ in (\ref{JumpDiffusionAdjoint}) is well defined, as the rest of Theorem \ref{Theorem:DualOperator-L*} follows directly from Proposition 4.6.1 in \cite{kunita2019stochastic}. 

For the simplicity of notation, we will only consider the case when $n=\ell=1$, $f$ does not depend on $t$, and we use $\cdot'$ (and $\cdot''$) to indicate taking the derivative of the concerned function in $x$. Under the current assumptions, it is easy to see that, for all $z\in\R$, $\lambda(t,x,z)$, $\gamma(t,x,z)$, $\gamma'(t,x,z)$ and $\gamma''(t,x,z)$ are all bounded by a constant multiple of $|z|\wedge 1$ with the constant being uniform in $(t,x)$ on any compact set. Thus,
\[ \pran{\phi_{t,z}^{-1}}'(x)=1-\gamma'(t,\phi_{t,z}^{-1}(x),z)+O(z^2\wedge 1)=1-\gamma'(t,x,z)+O(z^2\wedge 1),
\] 
where we used the fact that 
\[
\gamma'(t,\phi_{t,z}^{-1}(x),z)-\gamma'(t,x,z)=\gamma''(t,x,z)\lambda(t,x,z)+O(|z|^3\wedge 1).
\]
This observation, combined with the assumption that $\phi_{t,z}^{-1}$ is a diffeomorphism, implies that $\pran{\phi_{t,z}^{-1}}'(x)>0$ for all $z$. Similarly, we have
\[
\lambda(t,x,z)=\phi_{t,z}^{-1}(x)-x=\gamma(t,\phi_{t,z}^{-1}(x),z)=\gamma(t,x,z)+O(z^2\wedge 1),
\]
and hence 
\[f(\phi_{t,z}^{-1}(x))=f(x)+f'(x)\lambda(t,x,z)+O(z^2\wedge1)=f(x)+f'(x)\gamma(t,x,z)+O(z^2\wedge1).
\]
Therefore, for $|z|>1$, $\pran{\phi_{t,z}^{-1}}'(x)f(\phi_{t,z}^{-1}(x))-f(x)$ is bounded by a constant uniformly in $(t,x)$ on any compact set; for $|z|\leq1$,
\[
\pran{\phi_{t,z}^{-1}}'(x)f(\phi_{t,z}^{-1}(x))-f(x)-\gamma'(t,x,z)f(x)-\gamma(t,x,z)f'(x)=O(z^2).
\]
This is sufficient to guarantee that the integral with respect to $\nu$ in (\ref{JumpDiffusionAdjoint}) is well defined.
\end{proof}
\noindent \textbf{Remark.} It is clear from the derivation above that, in the jump-diffusion setting, the adjoint operator statement in Theorem \ref{Theorem:DualOperator-L*} relies critically on the assumption that $\phi_{t,z}$ is a diffeomorphism. A simple and classical case in which this assumption is satisfied is when $\phi_{t,z}$ is linear in $x$, meaning that $\gamma(x,t,z) = \gamma (t,z)$ does not depend on $x$. Indeed, in this case, the jump map $x\mapsto\phi_{t,z}(x)$ is just a translation by a constant, and thus immediately a $C^\infty$-diffeomorphism. 
Moreover, the adjoint operator reduces to
\begin{multline*}
    L^*f(t,x)= -\nabla \cdot (bf) (t,x)+\frac{1}{2} \sum_{i,j}\frac{\partial^2[(\sigma\sigma^T)_{ij}f]}{\partial x_i \partial x_j}(t,x) \\
    +\int_{\R^\ell} \brac{f(x-\gamma(t,z))-f(x) +\one_{|z|\leq 1}\nabla\cdot (\gamma f) (t,x,z)}\nu(dz).
\end{multline*}\\

As we wish to derive a system of equations analogous to \eqref{Equation:ItoSchrodingerSystem}, we require $\varphi,\hat\varphi$ to possess sufficient regularity, and this regularity may be inherited from the transition density function of the jump-diffusion process. Therefore, we turn our attention to the study of the transition density function.
We continue to proceed under the framework of \cite{kunita2019stochastic} and introduce some additional notions involving the jump component of the operator $L$. Let $\set{A_0(r)}_{0<r<1}$ be a family of \textit{star-shaped neighborhoods}\footnote{A collection of sets such that for every $0<r<1$, $A_0(r)\subseteq B_0(r)$ and $\bigcup_{r'<r}A_0(r')=A_0(r)$, where $B_0(r)$ denotes the ball of radius $r$ centered at the origin.} of the origin. We say that the L\'evy measure $\nu$ satisfies the \textit{order condition} with respect to $\set{A_0(r)}_{0<r<1}$ if for some $0<\alpha<2$, there exists a constant $c_\alpha> 0$ such that
\begin{equation*}
\int_{A_0(r)}|z|^2\nu(dz) \geq c_\alpha \,r^\alpha \text{ for all }0<r<1.
\end{equation*}
We also define the \textit{correlation matrix} $\Gamma_r$, $0<r<1$, to be the matrix with entries
\[
\Gamma^{i,j}_r=\frac{\int_{A_0(r)}z_iz_j\nu(dz)}{\int_{A_0(r)}|z|^2\nu(dz)},\,\,i,j=1,...,n,
\]
and set $\Gamma_0:=\lim_{r\downarrow 0}\Gamma_r$, assuming the limit exists. Finally, for $(t,x)\in[0,T]\times\R^n$, let $K(t,x)$ be the \textit{tangent vector field} of the jump map $\gamma$ at the origin, i.e., the columns of $K(t,x)$ are given by
\[
K_m(t,x)=\frac{\partial}{\partial z^m}\gamma(t,x,z)\bigg|_{z=0} ,\,\, m=1,...,\ell.
\]
\begin{thm}\cite[Theorem 6.5.2 and Theorem 6.6.1]{kunita2019stochastic}\label{Theorem:DensityKunita}
    Let $L$ be a jump-diffusion operator given by \eqref{ReferenceGenerator} with coefficients satisfying Assumptions \textup{\ref{Assumptions-(C)}} and $\set{X_t}_{[0,T]}$ be the solution to the SDE \eqref{ReferenceLevyIto} associated with $L$. Assume that there exists a family of star-shaped neighborhoods $\set{A_0(r)}_{0<r<1}$ with respect to which $\nu$ satisfies the order condition. Let $\Gamma_0$ and $K(t,x)$ be as above. Further suppose that $\sigma(t,x)\sigma^T(t,x)+K(t,x)\Gamma_0K^T(t,x)$ is positive definite for every $(t,x)\in[0,T]\times\R^n$.
    
    Then, $\set{X_t}_{[0,T]}$ admits a transition density function $p(s,x,t,y)$ for $0\leq s< t\leq T$ and $x,y\in \R^n$. Moreover,
    \begin{enumerate}
        \item for every $0\leq s<t\leq T$, $(x,y)\in\R^n\times\R^n\mapsto p(s,x,t,y)$ is a $C^{\infty,\infty}$-function, and for every multi-index $\mathbf{i}$ and every $x\in\R^n$, $y\mapsto\partial_x^{\mathbf{i}}p(s,x,t,y)$ is rapidly decreasing\footnote{That is, for any integer $k$ and any multi-index $\mathbf{j}$, $\abs{\partial_y^{\mathbf{j}}\partial_x^{\mathbf{i}}p(s,x,t,y)}(1+|y|)^k$ converges to $0$ as $|y|\to\infty$.};
        \item for every $(t,y)\in[0,T]\times\R^n$, $(s,x)\in (0,t)\times\R^n\mapsto p_{t,y}(s,x):=p(s,x,t,y)$ is a $C^{1,\infty}$-function that satisfies $\cL p_{t,y}(s,x)=0$, where $\cL:=\frac{\partial}{\partial s}+L$;
        \item for every $(s,x)\in[0,T]\times\R^n$,  $(t,y)\in(s,T)\times\R^n\mapsto p^{s,x}(t,y):=p(s,x,t,y)$ is a $C^{1,\infty}$-function that satisfies $\cL^*p^{s,x}(t,y)=0$, where $\cL^*:=-\frac{\partial}{\partial t}+L^*$ is the $L^2([0,T]\times\R^n)$-adjoint of $\cL$.
    \end{enumerate}
\end{thm}
In the above setting, we will say that the solution $\set{X_t}_{[0,T]}$ is a \textit{stochastic flow} generated by the SDE \eqref{ReferenceLevyIto}. An immediate consequence of Theorem \ref{Theorem:DensityKunita} is that \ref{Assumption:Hypo} holds in the current setting, which means that the results on the SBP established in Section \ref{Section:RelationSBP-htransform} are applicable. In other words, for a jump-diffusion operator $L$ whose coefficients satisfy Assumptions \ref{Assumptions-(C)}, the solution $\hat{\bP}$ to the SBP \eqref{Definition:DynamicSBP} associated with $L$ can always be interpreted as an $h$-transform of a reference path measure $\bP$ as $\hat{\bP}=\frac{\h(T,X_T)}{\h(0,X_0)}\bP
$ where 
\[
\h(t,x):=\int_{\R^n}\g(y)p(t,x,T,y)dy\text{ for }(t,x)\in[0,T]\times\R^n.
\]Moreover, $\hat{\bP}$ can always be achieved as the limit under strong convergence of a sequence of $h$-transforms of $\bP$ by harmonic functions.

As one can see, $\h$ coincides with $\varphi$ defined in \eqref{Equation:phi-definition}. Therefore, Theorem \ref{Theorem:DensityKunita} allows us to use flow-based approaches to study the SBP and to describe the dynamics of $\hat{\bP}$ in the setting with sufficient regularity. Indeed, the integro-differential operators $L,L^*$, as in \eqref{ReferenceGenerator} and \eqref{JumpDiffusionAdjoint}, correspond to the backward and forward equations respectively, and they yield the following system of PIDEs, which is a generalization of \eqref{Equation:ItoSchrodingerSystem}. 
\begin{thm}\label{THeorem:varphi_hatvarphi solution to Schrodinger system}
    Assume that the conditions of Theorem \ref{Theorem:DensityKunita} are satisfied, and suppose that the Schr\"odinger system \eqref{Equation:fgSchrodingerSystem} admits solutions $\f,\g$ that are non-negative, measurable and bounded functions with compact support. Then the functions $\varphi,\hat\varphi$ given by \eqref{Equation:phi-definition}, \eqref{Equation:phihat-definition} are of class $C^{1,\infty}$ and the pair $(\varphi,\hat\varphi)$ solves the system of PIDEs
\begin{equation}\label{Equation:LevySchrodingerSystem}
\begin{dcases}
\frac{\partial \varphi}{\partial t}(t,x)+L\varphi(t,x) =0\;\text{ for }(t,x)\in(0,T)\times\R^n,\\
\frac{\partial \hat\varphi}{\partial t}(t,x)-L^*\hat\varphi(t,x)=0\;\text{ for }(t,x)\in (0,T)\times\R^n,\\
\rho_0(x)=\varphi(0,x)\hat{\varphi}(0,x)\;\text{ for }x\in \R^n,\\
\rho_T(x)=\varphi(T,x)\hat{\varphi}(T,x)\;\text{ for }x\in \R^n.
\end{dcases}
\end{equation}
\end{thm}  
\noindent We omit the proof of this theorem, as all the statements follow from the hypothesis and Theorem \ref{Theorem:DensityKunita} in a straightforward way. 

Recall that in Section \ref{subsection:varphi_hat varphi via f_g} we have already derived the product-shaped density of the marginal distribution of $\hat{\bP}$ and that is $\hat\bP_t(x)=\varphi(t,x)\hat\varphi(t,x)$ for $(t,x)\in[0,T]\times\R^n$. In the situation when $\varphi,\hat\varphi$ are sufficiently regular, we can prove this expression of $\hat\bP_t(x)$ in an alternative way by showing that $\cL^{{\hat\bP}*}(\varphi\hat\varphi)=0$ where $\cL^{\hat\bP}$ is the operator associated with $\hat\bP$ and $\cL^{{\hat\bP}*}$ is the adjoint operator of $\cL^{\hat\bP}$. In the diffusion case, this is often done by expressing the coefficients of $\cL^{{\hat\bP}*}$ explicitly and verifying that $\varphi\hat\varphi$ solves the concerned equation directly. Here, we will take a different (and shorter) route of proof by taking advantage of the fact that $\cL^{\hat\bP}$ falls into the category of operators studied in Lemma \ref{Theorem:ProductRule}.    
\begin{thm}
   Let $(\varphi,\hat\varphi)$ be the same as in Theorem \ref{THeorem:varphi_hatvarphi solution to Schrodinger system}. Then, $\varphi$ is harmonic and satisfies the mean-value property \eqref{Equation:MeanValueProperty}. Define $\cA_\varphi$ and $\cL^\varphi$ as in \eqref{definition_A_h} and \eqref{Definition:L^h} (with $h=\varphi$) respectively and denote by $\cL^{\varphi*}$ the adjoint operator of $\cL^\varphi$. Then, $\cL^{\varphi*}(\varphi\hat\varphi)(t,x)=0$ for all $(t,x)\in \cA_\varphi$.
\end{thm}
\begin{proof}
The fact that $\varphi$ is harmonic and satisfies (\ref{Equation:MeanValueProperty}) follows directly from (\ref{Equation:phi-definition}) and Theorem \ref{THeorem:varphi_hatvarphi solution to Schrodinger system}. Since $\varphi=\h$ and $\hat\bP=\bP^\h$, we know that $\hat\bP$ coincides with the $h$-transform of the reference measure $\bP$ by $\varphi$. Thus, $\cL^{\hat\bP}=\cL^\varphi$, i.e., $\cL^\varphi$ is the operator associated with $\hat\bP$. Recall that $\cL^\varphi$ is well defined and nontrivial at $(t,x)\in\cA_\varphi$ where $\varphi(t,x)>0$. Observe that at such $(t,x)$, $\cL^{\varphi*}$ is also well defined by (\ref{JumpDiffusionAdjoint}).

Given any test function $\eta\in C^{1,\infty}_c(\cA_\varphi)$, by Lemma \ref{Theorem:ProductRule}, we know that $\cL^\varphi\eta=\frac{\cL(\eta\varphi)}{\varphi}$ on $\cA_\varphi$. Therefore, if we denote by $\inn{\cdot,\cdot}$ the $L^2$ inner product over $\cA_\varphi$, then
\begin{equation*}
\inn{\cL^\varphi\eta,\varphi\hat\varphi}=\inn{\frac{\cL(\eta \varphi)}{\varphi},\varphi\hat\varphi}
    =\inn{\cL(\eta \varphi),\hat\varphi}
    =\inn{\eta \varphi,\cL^*\hat\varphi}
    =0,
\end{equation*}
where we use the fact that $\cL^*\hat\varphi=0$. We conclude that $\cL^{\varphi*}(\varphi\hat\varphi)=0$ on $\cA_\varphi$.
\end{proof}

So far we have been working with the case when $\varphi,\hat\varphi$ can inherit the regularity from the transition density function $p(s,x,t,y)$, to which end we require $\f,\g$ to be bounded and compactly supported. In fact, the properties of $p(s,x,t,y)$ allow us to relax the conditions on $\f,\g$. For example, (1) if $\g$ has at most polynomial growth and $\f$ is such that $\f\cdot\bR_0$ decreases sufficiently fast with respect to $\partial_y^\mathbf{j}p(s,\cdot,t,y)$, then the conclusions in Theorem \ref{THeorem:varphi_hatvarphi solution to Schrodinger system} remain true; (2) if $p(s,x,t,y)$ satisfies heat-kernel-type estimates \cite{grzywny2019heat,chen2018heat,czq2016heatkernel,kim2018heat,chen2018perturbation}, the collection of $\f,\g$ that will produce regular $\varphi,\hat\varphi$ can be further expanded. 

Even if $\f,\g$ are completely general and $p(s,x,t,y)$ is only known to satisfy Theorem \ref{Theorem:DensityKunita}, we can carry out the approximation procedure illustrated  in Section \ref{Subsec:Approximation} and obtain $\varphi,\hat\varphi$ as the limits of $C^{1,\infty}$-functions. 
To be specific, for every $k\geq1$, we define, for $(t,x)\in[0,T]\times\R^n$,
\begin{equation*}
    \begin{dcases}
        \varphi_k(t,x):=\int_{B_0(k)}(\g(y)\wedge k)p(t,x,T,y)dy, \\
        \hat\varphi_k(t,x):=\int_{B_0(k)}(\f(w)\wedge k)p(0,w,t,x)\bR_0(w)dw.  
    \end{dcases}
\end{equation*}
By Theorem \ref{THeorem:varphi_hatvarphi solution to Schrodinger system}, $\varphi_k,\hat\varphi_k$ are $C^{1,\infty}$ functions and  $k\geq 1$,  $\cL\varphi_k=0, \cL^*\hat\varphi_k=0$ for every $k\geq 1$. Moreover, as a simple application of the monotone convergence theorem, $\varphi_k \xrightarrow{k\to\infty}\varphi$ and $\hat\varphi_k \xrightarrow{k\to\infty}\hat\varphi$ point-wise, and hence $\varphi_k(t,x)\hat\varphi_k(t,x)\xrightarrow{k\to\infty}\hat\bP_t(x)$ for every $(t,x)\in[0,T]\times\R^n$. Finally, applying Scheff\'e's lemma, we arrive at the statement that the sequence of $h$-transforms $\bP^{\varphi_k}:=\frac{\varphi_k(T,X_T)}{\varphi_k(0,X_0)}\bP$ converges to $\hat\bP$ in the strong topology as $k\to\infty$.\\

\noindent \textbf{Remarks.} We conclude this subsection with two remarks regarding the transition density function. First, the existence of a smooth transition density function and the conclusions in Theorem \ref{Theorem:DensityKunita} may be established under different sets of conditions. For example, it is shown in \cite{kunita2013nondegenerate} (Section 6) that if $a:=\sigma\sigma^T$ is positive definite everywhere, $\nu$ has a ``very small tail'' in the sense that $\int_{\R^\ell}|z|^p\nu(dz)<\infty$ for all $p\geq 2$ (but $\nu$ does not need to satisfy the order condition), and $\gamma,\nabla\gamma$ satisfy some growth bound, 
then the corresponding jump diffusion admits a smooth transition density function. H\"ormander-type conditions are also discussed in \cite{kunita2013nondegenerate}. 
The existence of a smooth transition density function is also established in \cite{kunita2011analysis} under a uniform ellipticity\footnote{The condition adopted in \cite{kunita2011analysis} is $ |a(t,x)|\geq \frac{c_0}{(1+|x|)^{n_0}}$ for some constants $c_0,n_0$ uniformly in $t$.} condition on $a$.

Lastly, we note that the PIDE representation \eqref{Equation:LevySchrodingerSystem} of the Schr\"odinger system applies in more general settings and does not require the existence of a (smooth) transition density function. For example, the simple Poisson jump process discussed in Section \ref{Section:Poisson-Example} clearly does not possess a transition density function, but the corresponding functions $\varphi(t,x)=e^{(\lambda-\mu)t}(\mu/\lambda)^x$ and $\hat\varphi(t,x)=e^{-\lambda t}\frac{(\lambda t)^x}{x!}$ are indeed solutions to the PIDEs
\begin{equation*}
\begin{cases}
    \cL \varphi(t,x)=\frac{\partial \varphi}{\partial t}(t,x) +\lambda \pran{\varphi(t,x+1)-\varphi(t,x)}=0,\\
    \cL^*\hat\varphi(t,x)=\frac{\partial \hat\varphi}{\partial t}(t,x)+ \lambda(\hat\varphi(t,x)-\hat\varphi(t,x-1)) = 0,
\end{cases}
\end{equation*}
and the product $\varphi(t,x)\hat\varphi(t,x)=\exp(-\mu t)\frac{(\mu t)^x}{x!}$ is the probability mass function of the Schr\"odinger bridge in that example for all $t\in[0,T]$, $x\in\N$.
\subsection{Jump Diffusions with Stable-Like Jump Component}
Finally, we will examine a model for which Assumptions \ref{Assumptions-(C)} may be relaxed and \ref{Assumption:Hypo} may be absent. For this part, we take the setup from \cite{czq2016heatkernel} and consider a jump-diffusion operator $L$ with an \textit{$\alpha$-stable-like} jump component, and that is, for $f\in C_c^\infty(\R^n)$ and $(t,x)\in\R^+\times\R^n$, 
\begin{equation}\label{equation: L mix type}
\begin{aligned}
    Lf(t,x) := b(t,x)\cdot \nabla f(x)&+\frac{1}{2} \sum_{i,j}a_{ij}(t,x)\frac{\partial^2f}{\partial x_i \partial x_j}(x) \\
    &+\int_{\R^n} \brac{f(x+\xi)-f(x)-\one_{|\xi|\leq 1}\nabla f(x)\cdot \xi}\frac{\kappa(t,x,\xi)}{|\xi|^{n+\alpha}}d\xi,
\end{aligned}
\end{equation}
where $0 < \alpha <2 $ and the coefficients satisfy:
\begin{enumerate}
    \item $b$ is non-negative, measurable and bounded,
    \item $a$ is non-negative, measurable, and there exist $c_1>0$ and $\beta \in (0,1)$ such that for all $t>0$ and $x,y\in\R^n$,
    \begin{equation*}
        \abs{a(t,x)-a(t,y)}\leq c_1 \abs{x-y}^\beta,
    \end{equation*}
    \item there exists $c_2\geq 1$ such that for all $t>0$ and $x\in\R^n$,
    \begin{equation*}
        \frac{1}{c_2}\text{Id}_{n\times n} \leq a(t,x) \leq c_2\text{Id}_{n\times n},
    \end{equation*}
    where $\text{Id}_{n\times n}$ is the $n\times n$ identity matrix, and
    \item $\kappa$ is non-negative, measurable and bounded; when $\alpha=1$, for every $t>0$, $x\in\R^n$ and $r>0$,
    \begin{equation*}
        \int_{|\xi|\leq r} \xi\kappa(t,x,\xi)d\xi=0.
    \end{equation*}
\end{enumerate}

We note that the expression for $L$ in \eqref{equation: L mix type} does not quite match \eqref{ReferenceGenerator}. Indeed, it is always possible to rewrite the jump component in \eqref{ReferenceGenerator} as
\begin{equation*}
    \int_{\R^n}\brac{f(x+\xi)-f(x)-\one_{|\xi| \leq 1}\xi \cdot \nabla f(x)} M(t,x,d\xi),
\end{equation*}
where $M(t,x,E):=\nu\pran{\set{z\in\R^\ell:\gamma(t,x,z)\in E}}$ for Borel $E\subseteq\R^n$, and adjust the drift term to account for the change in the integration region of small jumps \cite{kurtz2011equivalence}. In the case when $M(t,x,\cdot)$ is of the $\alpha$-stable-like form, i.e.,
\begin{equation*}
    M(t,x,d\xi)=\frac{\kappa(t,x,\xi)}{|\xi|^{n+\alpha}}d\xi,
\end{equation*}
the operator $L$ has been extensively studied in recent years \cite{grzywny2019heat,foondun2009heat,chen2018heat,chen2016heat,chen2016uniqueness,jin2017heat,czq2016heatkernel,bass2009regularity}. In particular, it is shown in \cite{czq2016heatkernel} (Theorem 1.1 and Theorem 4.1) that there exists a unique non-negative continuous function $p(s,x,t,y)$ such that $\set{p(s,x,t,y):0\leq s< t, \;x,y\in\R^n}$ forms a family of transition density functions\footnote{That is, for each $0\leq s\leq t$, $\int_{\R^n}p(s,x,t,y)dy=1$ and the family of $p(s,x,t,y)$ satisfies the Chapman-Kolmogorov equation.}, and if for $f\in C_c^2(\R^n)$, we define \[
    P_{s,t}f(x):=\int_{\R^n}f(y)p(s,x,t,y)dy,
\] 
then 
\[
    P_{s,t}f(x)=f(x)+\int_s^tP_{s,r}Lf(r,x)dr.
\]

Let us now examine the SBP for $L$ in (\ref{equation: L mix type}). It is clear that \ref{Assumption:MartingaleProblem-s,x} is satisfied in the current setting. However, we do not have sufficient information on the regularity of $p(s,x,t,y)$. Neither do we have any evidence to support \ref{Assumption:Hypo}. To overcome this obstacle, we again turn to an approximation approach and adopt the mollification method outlined in \cite{czq2016heatkernel}.

To be specific, we take a non-negative function $\psi\in C_c^\infty(\R^n)$ such that $\int_{\R^n}\psi=1$, and for $\eps>0$, set $\psi_\eps(\cdot):=\frac{1}{\eps^n}\psi(\frac{\cdot}{\eps})$. Further, we choose a sequence $\set{\eps_m : m \geq 1}\subseteq(0,1)$ that decreases to 0, and define, for every $m\geq1$, $b_m:=b\ast\psi_{\eps_m}$, $\kappa_m:=\kappa \ast\psi_{\eps_m}$ and denote by $L_m$ the operator in (\ref{equation: L mix type}) with $b,\kappa$ replaced by $b_m,\kappa_m$.
Then, following the proof of Theorem 4.1 in \cite{czq2016heatkernel}, we obtain that $L_m$ admits a transition density $p_m(s,x,t,y)$ which, as a function in $(s,x)$, is of class $C^{1,2}$, and hence \ref{Assumption:Hypo} holds for $L_m$. Moreover, as $m\to\infty$, $p_m(s,x,t,y)$ converges to $p(s,x,t,y)$ uniformly on every compact set along some subsequence.\\

We return to the original operator $L$ for a moment. Let $\bR$ be the reference measure associated with $L$, $\rho_0,\rho_T$ be target endpoint distributions, and $\bP$ be as in \eqref{eq:def of bP}. Assume $(\f,\g)$ is the solution to the Schr\"odinger system (\ref{Equation:fgSchrodingerSystem}) and define $\h$ as in (\ref{def:frak_h}). By repeating the procedure in Section \ref{Subsec:Approximation}, we can find a sequence $\set{g_k}_{k\geq1}\subseteq C_c^\infty(\R^n)$ such that $\norm{g_k}_{\text{u}}\leq k$ for every $k\geq1$, and if \[h_k(t,x):=\int_{\R^n}g_k(y)p(t,x,T,y)dy,\]then $h_k(t,x)\to\h(t,x)$ as $k\to\infty$ for every $(t,x)\in[0,T]\times\R^n$. Now we will use $\set{g_k}_{k\geq1}$ to define another sequence of functions: for every $m\geq1$, \[
h^{(m)}_k(t,x):=\int_{\R^n}g_k(y)p_m(t,x,T,y)dy.
\]Since \ref{Assumption:Hypo} holds for every $L_m$, $h_k^{(m)}$ is harmonic for every $m\geq1,k\geq1$. We will prove below that we can extract from $\set{h_k^{(m)}}_{k,m\geq1}$ an approximation of $\h$. 
\begin{lem}\label{lem:convergence of subsequence of h}
There exists a subsequence $\set{m_k:k\geq1}$ such that
\begin{equation*}
    \lim_{k\to\infty}h_k^{(m_k)}(t,x)=\h(t,x)\text{ for every }(t,x)\in[0,T]\times\R^n.
\end{equation*}
\end{lem}
\begin{proof}
    For every $(t,x)$, set $P_{t,T}^{(m)}(x,dy):=p_m(t,x,T,y)dy$. Since  $p_m\xrightarrow{m\to\infty} p$, we have $P_{t,T}^{(m)}(x,dy)\xrightarrow{m\to\infty}P_{t,T}(x,dy)$, and hence the sequence of measures $\set{P_{t,T}^{(m)}(x,\cdot)}_{m\geq 1}$ is tight. Let $\set{D_k : k \geq 1}$ be a sequence of compact sets in $\R^n$ such that
\begin{enumerate}
    \item $\sup_m P_{t,T}^{(m)}(x,D_k^c) \leq 1/k^3$, and
    \item $P_{t,T}(x,D_k^c) \leq 1/k^3$.
\end{enumerate}
Choose a subsequence $\set{m_k}_{k\geq1}$ such that for all $k\geq1$,
\begin{equation*}
     \abs{p_{m_k}(t,x,T,y)-p(t,x,T,y)}\leq \frac {1}{k^2 \text{vol}(D_k)}
\end{equation*}
for all $t\in[0,T], x\in \overline{B_0(k)}$ and $y\in D_k$. Therefore, for every $(t,x)\in[0,T]\times\R^n$,  
\begin{align*}
    \abs{h_k^{(m_k)}(t,x)-\h(t,x)}&\leq\abs{h_k(t,x)-\h(t,x)}+\abs{h_k^{(m_k)}(t,x)-h_k(t,x)}\\
    &=\abs{h_k(t,x)-\h(t,x)}+\abs{\int_{\R^n}g_k(y)(p_{m_k}(t,x,T,y)-p(t,x,T,y))dy},
\end{align*}
where the first term converges to 0 as $k\to \infty$, as discussed above,
and the second term is bounded from above by
\begin{align*}
    & \int_{D_k^c} \abs{g_k(y)}\abs{p_{m_k}(t,x,T,y)}dy+\int_{D_k^c} \abs{g_k(y)}\abs{p(t,x,T,y)}dy\\
    & \hspace{2in} +\int_{D_k} \abs{g_k(y)}\abs{p_{m_k}(t,x,T,y)dy-p(t,x,T,y)}dy\\
    \leq&\, k \cdot \frac{1}{k^3} + k \cdot \frac{1}{k^3}+ k \cdot \text{vol}(D_k) \cdot  \frac {1}{k^2 \text{vol}(D_k)}\,\text{  (according to the choice of $D_k,m_k$),}
\end{align*}
which also tends to $0$ as $k\to \infty$.   We conclude that $h_k^{(m_k)}(t,x)\xrightarrow{k\to\infty}\h(t,x)$ point-wise. 
\end{proof}
Finally, we are ready to interpret the solution to the SBP in this approximation context. For every $k\geq1$, by Assumption \ref{Assumption:Hypo}, $h_k^{(m_k)}$ is harmonic and satisfies the mean-value property \eqref{Equation:MeanValueProperty} with respect to $L_{m_k}$. Let $\bP_k$ be the solution to the martingale problem for $(L_{m_k},\rho_0)$, and $\set{X_t}_{[0,T]}$
be the canonical process under $\bP_k$. Define \begin{equation}\label{eq:definition of bQ_k}
  \bQ_k:=\frac{h_k^{(m_k)}(T,X_T)}{h_k^{(m_k)}(0,X_0)}\frac{\one_{\set{X_0\in A_k}}}{r_k}\bP_k,  
\end{equation} 
where $A_k:=\set{x\in\R^n:h_k^{(m_k)}(0,x)>0}$ and $r_k:=\rho_0(A_k)$. Then, $\bQ_k$ is the $h$-transform of $\bP_k$ by $h_k^{(m_k)}$. Note that $\bQ_k$ is not necessarily related to the SBP for $\bP_k$ with the original target endpoint distributions $(\rho_0,\rho_T)$, and that is because $h_k^{(m_k)}$ comes from $(\f,\g)$ that solves the Schr\"odinger system for $L$ rather than $L_{m_k}$. However, we will show that the sequence $\set{\bQ_k}_{k\geq 1}$ converges to $\hat\bP=\bP^\h$ in the sense of convergence in finite-dimensional distributions. In other words, we can achieve an approximation to the original SBP by considering the counterpart problem for the ``mollified'' version of the operator. 
\begin{thm}\label{Thm:approximation of P^ in finite dimensional dist}
     Assume that $\hat\bP$ is the solution to the original SBP \eqref{Definition:DynamicSBP} with the reference measure $\bR$ associated with the operator $L$ in \eqref{equation: L mix type} and target endpoint distributions $\rho_0,\rho_T$. Let $\bP,\f,\g,\h$ and $\set{\bP_k}_{k\geq1}$ be as described above. Choose the sequence  $\set{h_k^{(m_k)}}_{k\geq1}$ as in Lemma \ref{lem:convergence of subsequence of h} and define $\bQ_k$ as in \eqref{eq:definition of bQ_k} for every $k\geq1$. Then, the finite-dimensional distributions under $\bQ_k$ converge to those under $\hat\bP$ as $k\to \infty$.
 \end{thm}
 \begin{proof}
Fix $0 = t_0 \leq t_1 \leq \dots \leq t_N = T$, and $B_0,B_1,\dots,B_N \in \cB_{\R^n}$. We have 
     \begin{equation*}
     \begin{aligned}    \bQ_k\pran{\bigcap_{i=0}^N \set{X_{t_i}\in B_i}} = \frac{1}{r_k} \int_{B_0\cap A_k}&\dots\int_{B_N}\frac{h_k^{(m_k)}(T,x_N)}{h_k^{(m_k)}(0,x_0)}\\
     &\cdot\prod_{i=0}^{N-1}p_{m_k}(t_i,x_i,t_{i+1},x_{i+1})dx_Ndx_{N-1}\dots dx_1\rho_0(dx_0),
     \end{aligned}
     \end{equation*}
and 
     \begin{equation*}
     \begin{aligned}    \bP^{\h}\pran{\bigcap_{i=0}^N \set{X_{t_i}\in B_i}} = \int_{B_0}\dots\int_{B_N}&\frac{\h(T,x_N)}{\h(0,x_0)}\\
     &\cdot\prod_{i=0}^{N-1}p(t_i,x_i,t_{i+1},x_{i+1})dx_Ndx_{N-1}\dots dx_1\rho_0(dx_0).
     \end{aligned}
     \end{equation*}
Since $h_k^{(m_k)}$ satisfies the mean-value property and converges to $\h$ as $k\to\infty$, we can mimic the proof of Theorem \ref{thm:P^hk strong conv to hatP} to show that $\lim_{k\to\infty}r_k=1$. Therefore, by the previous lemma, 
    \begin{equation*}
        \frac{\one_{A_k}(x_0)}{r_k}\frac{h_k^{(m_k)}(T,x_N)}{h_k^{(m_k)}(0,x_0)}\xrightarrow{k\to\infty}\frac{\h(T,x_N)}{\h(0,x_0)}
    \end{equation*}
Moreover, $p_{m_k}(t_i,x_i,t_{i+1},x_{i+1})\xrightarrow{k\to\infty}p(t_i,x_i,t_{i+1},x_{i+1})\,\,\text{ for }i=0,\dots,N-1$. Consequently,  
\begin{equation*}
    \lim_{k\to\infty}\bQ_k\pran{\bigcap_{i=0}^N \set{X_{t_i}\in B_i}} =\bP^{\h}\pran{\bigcap_{i=0}^N \set{X_{t_i}\in B_i}}=\hat\bP\pran{\bigcap_{i=0}^N \set{X_{t_i}\in B_i}}.
\end{equation*}
\end{proof}
In conclusion, by mollifying the operator $L$, the solution $\hat\bP$ to the original SBP is the limit of a sequence of path measures $\set{\bQ_k:k\geq 1}$ in the sense of convergence of finite-dimensional distributions, where for each $k\geq 1$, $\bQ_k$ is the $h$-transform, by a harmonic function, of the path measure associated with the mollified version of $L$.
\section{Disclosure of interest}
The authors report there are no competing interests to declare. \footnote{Word count: 11148.}
\bibliographystyle{plain} 
\bibliography{main-refs}  
\end{document}